\documentclass[reqno,12pt]{amsart}

\usepackage{defs_KLM}
\begin{document}

\title[General Multivariate Hawkes Processes]{Exact and Asymptotic Analysis of General Multivariate Hawkes Processes and Induced Population Processes}

\author{Raviar Karim, Roger J.~A. Laeven, and Michel Mandjes}

\begin{abstract}


This paper considers population processes in which general, not necessarily Markovian, multivariate Hawkes processes dictate the stochastic arrivals.
We establish results to determine the corresponding time-dependent joint probability distribution,
allowing for general intensity decay functions, general intensity jumps, and general sojourn times.
We obtain an exact, full characterization of the time-dependent joint transform of the multivariate population process and its underlying intensity process
in terms of a fixed-point representation and corresponding convergence results.
We also derive the asymptotic tail behavior of the population process and its underlying intensity process
in the setting of heavy-tailed intensity jumps.
By exploiting the results we establish, arbitrary joint spatial-temporal moments and other distributional properties can now be readily evaluated
using standard transform differentiation and inversion techniques,
and we illustrate this in a few examples.

\vb


\vb



\vb

\noindent
{\sc Affiliations.}
RK and RL are with the Dept. of Quantitative Economics, University of Amsterdam.
RL is also with E{\sc urandom}, Eindhoven University of Technology, and with C{\sc ent}ER, Tilburg University;
his research was funded in part by the Netherlands Organization for Scientific Research under grants NWO VIDI 2009 and NWO VICI 2019/20.
MM is with the Korteweg-de Vries Institute for Mathematics, University of Amsterdam. 
MM is also with E{\sc urandom}, Eindhoven University of Technology, 
and with the Amsterdam Business School, University of Amsterdam; 
his research was funded in part the Netherlands Organization for Scientific Research under the Gravitation project N{\sc etworks}, grant 024.002.003.
Email addresses: \tt{r.s.karim@uva.nl}, \tt{r.j.a.laeven@uva.nl}, \tt{m.r.h.mandjes@uva.nl}.

\vb

\noindent
{\it Date}: {\today}.

\end{abstract}

\maketitle

\newpage

\section{Introduction}


As the world grows more interconnected,
event shocks tend to spread and cluster more easily and more forcefully.
Prototypical examples include
the contagious spread of diseases across populations,
financial contagion across equity or credit markets,
and cyber infections across technological networks.
The amplification of event shocks over time and in space, that is, 
across populations, markets and networks,
arguably constitutes one of the core challenges to modern risk measurement and risk management. 

In principle, 
multivariate point processes can provide a probabilistic description of the occurrence of events,
and their stochastic dependencies, in time and space.
Among multivariate point processes, the class of Hawkes processes (\cite{H71}),
or mutually exciting processes,
provides a natural contender for modeling contagious phenomena.
Different from multivariate Poisson, or more generally L\'evy, processes,
they allow for clustering to occur not just in the spatial (i.e., cross-sectional) dimension,
but also in the temporal (i.e., time-series) dimension.
Originally introduced to stochastically describe epidemics and earthquake occurrences,
Hawkes processes have seen increased interest over the past few years,
in finance (\cite{ACL15,ALP14,BMM15,EGG10,H17}),
social interaction (\cite{CS08,RLMX17}),
neuroscience and genome analysis (\cite{B88,RS10}), and so on.

Analyzing distributional properties of general multivariate Hawkes processes is, however, challenging.
The existing literature often focuses on
the special case of exponential excitation functions under which the Hawkes process
(more precisely, the vector consisting of the counting process and its intensity process jointly)
can be shown to be Markovian
(\cite{ACL15,ALP14,DZ11,DP18,DP20,EGG10,H71,O75}),
or on other specific cases or asymptotic regimes (discussed below in more detail).

In this paper, we establish exact and asymptotic results on distributional properties of general---in our context, not necessarily Markovian---multivariate Hawkes processes
and related population processes, with widespread use across various applications.
In a series of exact results, we obtain a full characterization of the 
joint transform
of the multivariate population process and its underlying arrival intensity process,
in terms of a fixed-point representation and corresponding convergence results.
From these results, arbitrary joint moments, including auto- and cross-covariances, and other distributional properties such as joint event probabilities,
can be readily obtained
using standard transform differentiation and inversion techniques.

Our exact results exploit a cluster representation of the Hawkes process,
for the univariate self-exciting process first described in \cite{HO74};
see also \cite{DVJ07}.
The cross-excitation phenomenon that is present in multivariate Hawkes processes
significantly complicates the situation 
compared to the univariate setting 
where only self-excitation is present.
More specifically, cross-excitation leads to clusters with branches (i.e., offspring)
in the time-series as well as in the cross-sectional dimensions,
generating complex and highly intertwined clusters. 
Our analysis of a $d$-dimensional population process whose arrivals are
described by a general multivariate Hawkes process,
yields a fixed-point theorem that characterizes the joint transform of the respective processes
at, potentially multiple, future time points,
involving suitably defined $d\times d$-dimensional objects
to represent the full range of cross-sectional and time-series dependencies.

In a series of asymptotic results, we characterize the tail behavior of our population process
and the underlying intensity process in a setting of heavy-tailed intensity jumps.
These results pertain directly to the respective probability distributions 
and enable us to derive associated tail probabilities.
We also establish several \textit{class properties}, and irreducibility results, using the nomenclature of Markov chains.

Intuitively, when analyzing the asymptotic tail behavior in our general multivariate setting,
one might expect that the heaviest tail among the tails of the distributions of intensity jumps that excite component $i$ `dominates',
and therefore dictates the 
tail behavior of component $i$.
Our asymptotic results reveal that this intuition is not necessarily true as,
due to the cross-excitation phenomenon, heavy tails originating in different components may propagate to component $i$ indirectly,
through other components in the system.
From the full representation of the mutually exciting behavior within the system as provided by our fixed-point theorem,
along with suitable Tauberian theorems (\cite{BGT87}), we derive a system of vector-valued renewal equations
that jointly characterize the asymptotic tail behavior of our population process and the underlying intensity process.
Both our exact and asymptotic results are directly amenable to numerical evaluation
and we illustrate our results in a 
collection of numerical examples.

This paper relates to several branches in the existing literature, which we categorize along three dimensions.
First, a 
starting point for our analysis is provided by the branching structure that underlies a Poissonian cluster representation of multivariate
Hawkes processes.
In the univariate self-exciting setting, this branching structure was first discussed in \cite{HO74};
see also e.g., the formal and extensive treatment in \cite{DVJ07,BMM15}, and \cite[Chapter 4]{L09}.
We use a cluster representation for the complex intertwined spatial-temporal structure of general multivariate Hawkes processes
to obtain suitable general and micro-level distributional equalities for our population process
and its underlying intensity process,
which we next exploit to characterize the joint transform.


Second, transform analysis and the derivation of transient and stationary moments for Hawkes processes
has received considerable attention in the literature, especially under Markovian assumptions.
In a univariate Markov setting, \cite{DZ11,FZ14} characterize the joint transform of the point process and the underlying intensity process
by relying on the infinitesimal Markov generator,
yielding systems of ODEs for the moments of both processes; see also \cite{DP18,DP20,KSBM18}.
In addition, \cite{KSBM18} and \cite{GZZ18} 
consider the case of non-exponential decay for the probability generating function of the univariate point process
and its joint Laplace transform, respectively.
In a multivariate Markov setting, and more generally in the context of (Markovian) affine point processes,
\cite{DPS00} provide semi-analytic expressions of conditional characteristic functions,
involving solutions to systems of ODEs; see also \cite{ALP14,EGG10}.
Closed-form expressions of stationary moments as Taylor series approximations over short time intervals in a multivariate Markov setting
are derived in \cite{ACL15}
by exploiting operator methods. 
We characterize the joint transform for general, possibly non-Markovian, multivariate Hawkes processes and induced population processes, at possibly multiple time points,
through a fixed-point representation,
allowing for general decay functions, general distributions of the intensity jump sizes, and general distributions of the sojourn times.
We are not aware of other work on exact transform and moment characterizations for Hawkes processes that allows for a comparable degree of generality along all these dimensions.


Third, asymptotic results such as LLNs and FCLTs for multivariate Hawkes processes have been established in e.g.,
\cite{BDHM13,GZ18-3}.
Furthermore, the nearly unstable situation is analyzed in \cite{JR15,JR16}
and the setting of a large initial intensity is considered in \cite{GZ18}.
Large and precise deviation results are obtained in e.g., \cite{BT07,GZ21} for large times and general univariate Hawkes processes
and in \cite{GZ18-2} for a large initial intensity in the Markov case.
These large and precise deviation results are obtained in a setting of light-tailed counting and intensity processes.
We study the asymptotic tail behavior of the general
population process and its underlying intensity process
in the setting of \textit{heavy-tailed} intensity jumps.


We finally note that the population process analyzed in this paper may be naturally interpreted and applied in the---now highly relevant---context of epidemiological modeling.
It provides an appealing probabilistic description of the contagious amplification of viruses among populations across the globe.
Recent work that uses Hawkes processes, potentially in conjunction with SIR-models, to stochastically describe pandemics such as COVID-19 includes
\cite{CGM18,C20}.


The remainder of this paper is organized as follows.
In Section \ref{sec:ModelBranching}, we define the general multivariate Hawkes process, describe its cluster representation, and discuss some of its properties.
In Section \ref{sec:JointTransforms}, we exploit the branching structure to obtain distributional equalities, which we use to characterize the joint transform of the 
process under consideration.
In Section \ref{sec:FixedPoint}, we represent the joint transform as the fixed point of a certain mapping and establish corresponding convergence results.
In Section \ref{sec:TailProbs}, we derive the asymptotic tail behavior.
Section \ref{sec:Numerics} contains our numerical illustrations.
Conclusions are in Section \ref{sec:Conclusion}.
The proofs of some auxiliary results are relegated to the Appendix.


\section{Model and Underlying Branching Structure}\label{sec:ModelBranching}

In this section, we provide a formal definition and a cluster representation of general multivariate Hawkes processes,
and discuss their properties and branching structure.
Throughout, we adopt the notation ${\bm x} =(x_1,\ldots,x_d)^\top$, for a given value of $d\in{\mathbb N}$.

\subsection{Definition and properties}


Hawkes processes constitute a general class of multivariate point processes.
In full generality, we consider a $d$-dimensional c\`adl\`ag point process ${\bm N}(\cdot)\equiv ({\bm N}(t))_{t\in{\mathbb R}}$,
where each increment $N_i(t) - N_i(s)$ records the number of points in component $i\in[d]:=\{1,\dots,d\}$
in the time interval $(s,t]$, with $s<t$.
The points 
will be referred to as \textit{events}, where each event consists of a tuple $(T_r,i)$
that specifies the time of occurrence $T_r\in\rr$
of the $r$-th event and the component $i\in[d]$ in which it takes place.
As is well-known, a point process $\bm{N}(\cdot)$ can be characterized by its conditional intensity $\bm{\lambda}(\cdot)$.
The $i$-th component of $\bm{\lambda}(t)$, $i\in[d]$, is given by
\begin{align}
\label{eq: definition general conditional intensity}
    \lambda_i(t) = \lim_{h\downarrow 0} \frac{\ee\big[ N_i(t+h) - N_i(t) \, | \, \cF_{t-}\big]}{h},
\end{align}
where $\cF_{t-}=\sigma(\bm{N}(s):s<t)$ is the sigma algebra of events up to, but not including, time $t$; see e.g.,\ \cite[Chapter 7]{DVJ07}.
We refer to $\bm{\lambda}(\cdot)$ simply as the \textit{intensity}, where it is noted
that it may itself be a stochastic process.
Clearly, $\bm{\lambda}(\cdot)$ is predictable.

\begin{definition}
\label{def: hawkes intensity}
A general {\em multivariate Hawkes process} $($\cite{H71}$)$ is a point process $\bm{N}(\cdot)$ whose components $N_i(\cdot)$, for $i\in[d]$, satisfy
\begin{align}
\label{eq: Hawkes intensity def}
    \begin{cases}
    \pp(N_i(t+\Delta) - N_i(t) = 0 \, | \, \cF_t) = 1 - \lambda_i(t)\Delta + o(\Delta),\\
    \pp(N_i(t+\Delta) - N_i(t) = 1 \, | \, \cF_t) = \lambda_i(t)\Delta + o(\Delta), \\
    \pp(N_i(t+\Delta) - N_i(t) > 1  \, | \, \cF_t) = o(\Delta),
    \end{cases}
\end{align}
as $\Delta \downarrow 0$.
Here, $\cF_t=\sigma(\bm{N}(s):s\leqslant t)$  
is the natural filtration generated by $\bm{N}(\cdot)$.
With $\lambda_{i,\infty}>0$ and $g_{ij}(\cdot)$ non-negative integrable functions,
the intensity $\lambda_i(\cdot)$ takes the form
\begin{align}
\label{eq: intensity lambda def}
    \lambda_i(t)
    &= \lambda_{i,\infty}+ \sum_{j=1}^d  \int_{-\infty}^t B_{ij}(s)\, g_{ij}(t-s)\,\mathrm{d}N_j(s),
\end{align}
where, for each $i,j\in[d]$, the $(B_{ij}(s))_s$ constitutes a sequence of cross-sectionally and serially independently distributed random variables
that are distributed as the generic non-negative random variable $B_{ij}$.
\end{definition}

Informally, the definition above is understood as follows.
The constant $\lambda_{i,\infty}$ represents the {base rate} corresponding to component $i$.
An event generated by $N_j(\cdot)$ in component $j$
leads to a jump in the intensity $\lambda_i(\cdot)$ of component $i$, with $i,j\in[d]$;
its size is distributed as the random variable $B_{ij}$.
After the occurrence of this event, the decay functions $g_{ij}(\cdot)$ govern the path of the intensity $\lambda_i(\cdot)$
back to the base rate $\lambda_{i,\infty}$.

The multivariate Hawkes process $\bm{N}(\cdot)$ is also known as a \textit{mutually exciting} point process.
When an event in component $i$ impacts the intensity of component $i$, we speak of \textit{self-excitation}---a purely temporal effect.
When an event in component $j$ impacts the intensity of component $i$, with $i\neq j$, we speak of
\textit{cross-excitation}---a temporal as well as spatial effect.
Mutually exciting point processes accommodate both effects.

One may introduce functions $h_{ij}(\cdot)= B_{ij}\,g_{ij}(\cdot)$,
where the $B_{ij}$ is understood to be sampled at every event in $N_j(\cdot)$ in the manner described in Definition \ref{def: hawkes intensity}.
These $h_{ij}(\cdot)$, frequently called excitation or impact functions (see e.g.,\ \cite{H71,L09}),
couple the jump size and the decay function in a multiplicative manner.
We can thus compactly rewrite (\ref{eq: intensity lambda def}) in vector notation by setting
\begin{align}
\label{eq: intensity vector notation}
    \bm{\lambda}(t) = \bm{\lambda}_\infty + \int_{-\infty}^t \bm{H}(t-s)\, {\rm d}\bm{N}(s),
\end{align}
where $\bm{H}(\cdot) = (h_{ij}(\cdot))_{i,j\in[d]}$.
The matrix form $\bm{H}(\cdot)$ in (\ref{eq: intensity vector notation}),
where for each $j\in[d]$ we define the random vector $\bm{H}^j(\cdot) = (h_{1j}(\cdot),\dots,h_{dj}(\cdot))^{\top}$,
justifies the indexing convention of $h_{ij}(\cdot)$:
$h_{ij}(\cdot)$ describes the impact of events in \textit{source} component $j$
on the intensity $\lambda_i(t)$ of \textit{target} component $i$.
We note that the corresponding random vectors $\bm{B}^j$ could be considered as
mark random vectors associated to each event in component $j$,
such that tuples of the form $(T_r,j,\bm{B}^j_r)$ constitute a \textit{multivariate marked Hawkes process};
see \cite{L09} and \cite[Chapter 6.4]{DVJ07}.
Note that in the special case that $B_{ij} \equiv 0$ for all $i,j\in[d]$,
$\lambda_i(\cdot) \equiv \lambda_{i,\infty}$ such that $\bm{N}(\cdot)$ reduces to a $d$-dimensional homogeneous Poisson process with rate $\bm{\lambda}_{\infty}$.

With $(B_{ij,r})_r$ a sequence of i.i.d.\ random variables distributed as the generic $B_{ij}$,
observe that $\lambda_i(t)$ can be completely expressed in terms of the aggregated impact of the $B_{ij,r}$ due to all past events.
Indeed, considering the tuples $\{(T_r,k_r)\}_{r=1}^{R(t)}$, where $k_r \in [d]$ is the component of the event and $R(t) \in \nn$ is the total number of events strictly prior time $t$,
we can rewrite (\ref{eq: intensity lambda def}) as
\begin{align}
\label{eq: intensity lambda sum notation}
    \lambda_i(t) = \lambda_{i,\infty} + \sum_{r=1}^{R(t)} B_{ik_r,r}g_{ik_r}(t-T_r).
\end{align}

Existence, uniqueness and positivity of the intensity $\lambda_i(t)$ is guaranteed if $g_{ij}(\cdot)$ satisfies the conditions given in Definition \ref{def: hawkes intensity}
and $B_{ij} < \infty$ with probability one; see e.g.,\ \cite[Example 7.2(b)]{DVJ07}.
To guarantee stationarity of the Hawkes process, \cite{H71} shows that a stability condition must be imposed.
In the current multivariate setting, this condition takes the form $\rho(\lVert \bm{H} \rVert) < 1$, where $\lVert \bm{H}\rVert = (\lVert h_{ij} \rVert)_{i,j\in[d]}$ with
\[\lVert h_{ij} \rVert = \ee[B_{ij}] \cdot\lVert g_{ij} \rVert_{L^1(\rr_+)} = \ee[B_{ij}] \int_0^\infty g_{ij}(t)\ddiff t,\]
and $\rho(\cdot)$ denotes the spectral radius; see also \cite{BM96}.
In this case, as shown in \cite{H71}, the entries of the expected stationary intensity vector are the constants
$\lambda_i := \ee[\lambda_i(t)] = \ee[\ddiff N_i(t)]/\ddiff t$.
In vector form these intensities can be expressed as $\bm{\lambda} = (\bm{I} - \lVert \bm{H} \rVert)^{-1} \bm{\lambda}_{\infty}$.
In the sequel, we assume the stability condition applies.

While the methodology developed in this work applies to general non-negative integrable functions $g_{ij}(\cdot)$,
a parametrization of special interest is that of exponential decay.
\begin{example}[\sc Exponential]
Let the decay functions $g_{ij}(\cdot)$ be of exponential form $g_{ij}(t) = e^{-\alpha_{ij}t}$,
for some $\alpha_{ij} > 0$ known as the decay rate.
Eqn.\ \eqref{eq: intensity lambda def} is in this case given by
\begin{align}
\label{eq: intensity lambda exponential decay}
    \lambda_i(t) = \lambda_{i,\infty}+ \sum_{j=1}^d  \int_{-\infty}^t B_{ij}(s) e^{-\alpha_{ij}(t-s)}\mathrm{d}N_j(s),
\end{align}
and can, by It\^{o}'s Lemma, alternatively be expressed in SDE notation as
\begin{align}
    \ddiff\lambda_i(t) = \sum_{j=1}^d \alpha_{ij}(\lambda_{i,\infty} -\lambda_i(t))\ddiff t + \sum_{j=1}^d B_{ij}(t) \,\ddiff N_j(t).
\end{align}
A distinctive property of exponential decay is that the joint process $(\bm{N}(\cdot),\bm{\lambda}(\cdot))$ constitutes a Markov process;
see \cite{L09,O75}.
\end{example}
The Markov property yields a number of useful tools to analyze distributional properties of the Hawkes process in the case that $g_{ij}(\cdot)$ is of exponential form.
The explicit treatment in \cite{ACL15} relies on operator methods applied to the Markov infinitesimal generator.
In \cite{FZ14,DZ11}, the Markov infinitesimal generator is used to characterize the conditional joint transform of $(N(\cdot),\lambda(\cdot))$ as the solution to a system of ODEs;
see also \cite{DPS00}.
A similar characterization is given in \cite{KSBM18}, which exploits the Markov property directly.

Departing from exponential decay renders the Hawkes process to be non-Markov in general.
An important example, extensively used to model e.g.,\ the temporal clustering of earthquake occurrences,
is the power-law parametrization proposed in \cite{O88}.

\begin{example}[\sc Power]
Consider the decay functions $g_{ij}(\cdot)$ to be of power-law type
by setting $g_{ij}(t) = 1/(c_{ij} + t)^{p_{ij}}$ for some $c_{ij} \in \rr_+$ and $p_{ij} > 1$.
Note that $p_{ij}>1$ ensures integrability.
In this case, by Eqn.\ \eqref{eq: intensity lambda def},
\begin{align}
\label{eq: intensity lambda power-law decay}
    \lambda_i(t)= \lambda_{i,\infty}+ \sum_{j=1}^d \int_{-\infty}^t \frac{B_{ij}(s)}{(c_{ij} + t-s)^{p_{ij}}}\,\mathrm{d}N_j(s).
\end{align}
\end{example}

Analyzing distributional properties of the Hawkes process with power-law decay is considerably more complex
than under exponential decay, due to the system being non-Markovian.
From an applications point of view, the power-law type and related parametrizations are of significant interest,
as they enable to model multivariate dynamic behavior that exhibits long-memory properties across time and space.

The Hawkes process can be used in conjunction with other models, such as (affine) jump-diffusion models \cite{ACL15,EGG10}
or epidemiological models \cite{CGM18,C20}.
In this paper, we consider a model that
also allows for \textit{departures}.
Considering the events generated by $\bm{N}(\cdot)$ as \textit{arrivals},
we introduce the \textit{Hawkes population process} $\bm{Q}(\cdot)$ by setting for $t\in\rr$
\begin{align}
\label{eq: definition queue arrival-depart}
    \bm{Q}(t) := \bm{N}(t) - \bm{D}(t),
\end{align}
where $\bm{D}(\cdot)$ is a point process such that $D_j(t)$ records the number of departures in component $j$ among the arrivals up to and including time $t$.
In what follows, we assume that all events $(T_r,j)$
have \textit{sojourn times} (i.e., times spent in component $j$)
that form a sequence of i.i.d.\ random variables
distributed as the non-negative random variable $J_j$.
Of course, if $J_{j}\equiv\infty$ for all $j\in[d]$, 
then $\bm{D}(\cdot) \equiv 0$, and hence $\bm{Q}(\cdot) = \bm{N}(\cdot)$.

Two examples of population processes occur in demography and epidemiology.
First, suppose $Q_i(\cdot)$ represents the number of (living or active) people in population $i\in[d]$
and $D_i(\cdot)$ the number of deaths.
Assuming that the individual lifetimes $J_i$ are exponentially distributed with mean $\mu_i^{-1}$ for some $\mu_i\geqslant 0$,
the process $D_i(\cdot)$ is an inhomogeneous Poisson process with rate $\mu_i Q_i(t)$; cf.\ \cite{ABGK93}.
Second, $Q_i(\cdot)$ may represent the number of infected people in geographic location $i\in[d]$
and $D_i(\cdot)$ the recovery process.

\begin{figure}
\includegraphics[scale=.52]{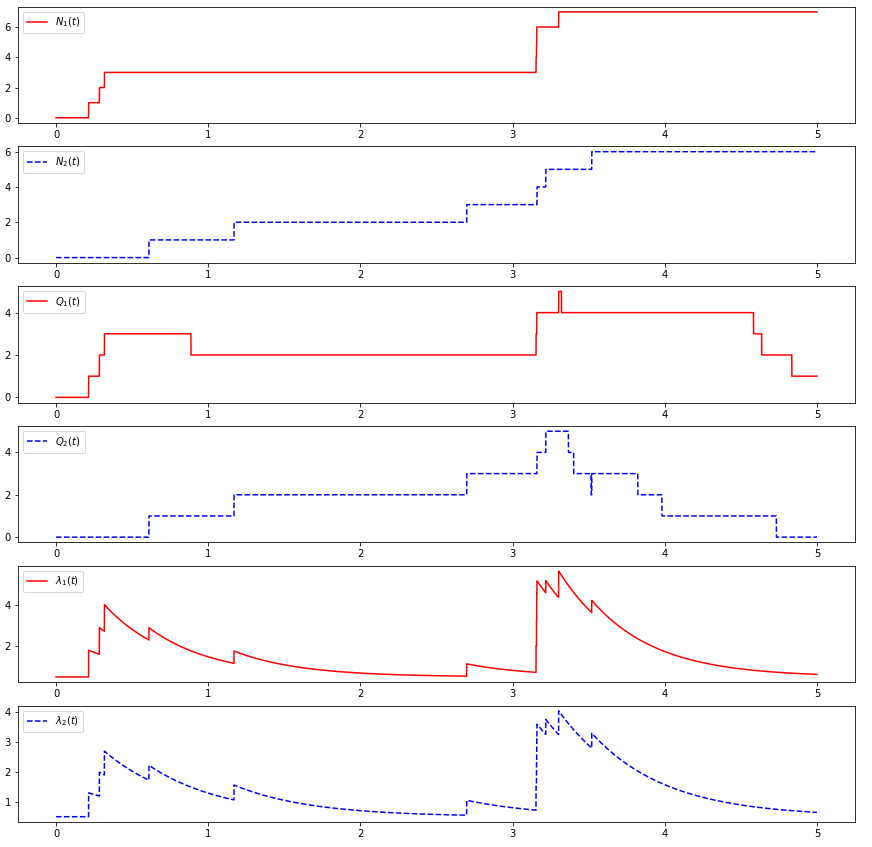}
\caption{\small \textit{Bivariate model: Sample paths of
$\bm{N}(\cdot)=(N_1(\cdot),N_2(\cdot))$,
$\bm{Q}(\cdot) = (Q_1(\cdot),Q_2(\cdot))$ and
$\bm{\lambda}(\cdot)=(\lambda_1(\cdot),\lambda_2(\cdot))$
under exponential decay and exponentially distributed sojourn times.
Parameters: $\lambda_{1,\infty} = \lambda_{2,\infty} = 0.5$, $\mu_1 = \mu_2 =1$,
$\alpha_{11}=\alpha_{12} = 2.3$, $\alpha_{21} = \alpha_{22} = 2$, $B_{11} \equiv 1.3$, $B_{12} \equiv 0.6$, $B_{21} \equiv 0.8$, $B_{22} \equiv 0.5$.}}
\label{fig: bivariate sample path distr eq}
\end{figure}

In Figure \ref{fig: bivariate sample path distr eq}, we plot a sample path of our model when $d=2$.
One readily observes that each event,
irrespective of which component serves as the source, leads to jumps in both intensities,
depending upon the $B_{ij}$'s involved, as described by \eqref{eq: intensity vector notation}.
Excited intensities increase the likelihood of future events and
it is clearly visibly that events in $\bm{N}(\cdot)$ are clustered in both time and space.
Note furthermore that the departures occurring in $\bm{Q}(\cdot)$ do not impact the intensities.

\subsection{Cluster representation and branching structure}

An interesting approach to represent the Hawkes process,
which will play a pivotal role in our analysis, is as a Poisson cluster process.
This so-called cluster representation is first described in \cite{HO74}
in the setting of the conventional single-dimensional Hawkes process;
see also \cite[Example 6.3(c)]{DVJ07} and \cite[Ch.\ IV]{L09}.
It essentially consists of \textit{immigrant} processes that generate events from the base rates
and \textit{cluster} processes that generate events occurring due to self- and cross-excitation.
In our multivariate setting, the cluster representation can be described as follows.

\begin{definition}[\textsc{Cluster representation}]
\label{def: hawkes cluster}
On the interval $[0,\tau]$ with $\tau \in \rr_+ \cup \{\infty\}$,
we define a point process $\bm{N}(\cdot)$
through a sequence of events 
generated according to the following procedure:
\begin{enumerate}
    \item[{\em (1)}] For each $j \in [d]$, let $I_j(\cdot)$ be a homogeneous Poisson process with intensity $\lambda_{j,\infty}$
    that generates immigrant events \[\left\{(T^{(0)}_r,j)\right\}_{r=1}^{R_j^{(0)}(t)},\]
    where $R_j^{(0)}(t)$ is the total number of immigrant events in component $j$ strictly prior to time $t\in[0,\tau]$.
    \item[{\em (2)}]Next, each immigrant event $(T^{(0)}_r,j)$ generates
    a $d$-dimensional cluster process 
    consisting of generations of descendants:
    \begin{enumerate}
        \item[{\em (a)}] $(T_r^{(0)},j)$ generates first-generation events \[\left\{(T_r^{(1)},m)\right\}_{r=1}^{R_m^{(1)}(t)}\] in each target component $m\in[d]$
        according to $K_{mj}(t-T_r^{(0)})$, where $K_{mj}(\cdot)$ is an inhomogeneous Poisson process with intensity $B_{mj,r}g_{mj}(\cdot)$,
        given the jump size $B_{mj,r}$ associated to $(T_r^{(0)},j)$.
        \item[{\em (b)}] Upon iterating (a) above, given the $r$-th event of the $(n-1)$-th generation in source component $m\in[d]$,
        descendant $(T_r^{(n-1)},m)$ generates $n$-th generation events \[\left\{(T_{r}^{(n)},l)\right\}_{r=1}^{R_l^{(n)}(t)}\]
        in each target component $l\in[d]$ according to $K_{lm}(t-T^{(n-1)}_r)$.
    \end{enumerate}
\end{enumerate}
By taking the Cartesian product over all components
and the union over all generations---immigrants and their descendants---,
we have that
\begin{align}
\label{eq: events union hawkes cluster}
    \bm{N}(t) =\bigcup_{n=0}^\infty \Big(\{(T_r^{(n)},1)\}_{r=1}^{R_1^{(n)}(t)} \times \cdots \times \{(T_r^{(n)},d)\}_{r=1}^{R_d^{(n)}(t)}\Big),
\end{align}
constitutes a general $d$-dimensional Hawkes process.
\end{definition}

The cluster representation in Definition~\ref{def: hawkes cluster}
agrees with Definition~\ref{def: hawkes intensity}
provided the stability condition is satisfied,
as shown in \cite{HO74,DVJ07}.
The stability condition has a natural interpretation in the context of the cluster representation:
it requires that each individual event generates a.s.\ finitely many descendants; see \cite[Example 8.3(c)]{DVJ07}.
The proof in \cite{HO74} amounts to comparing the cluster representation to an age-dependent birth-death process
allowing for immigration, and where the death process is set to zero.
Hence, we can naturally extend the cluster representation to our population process $\bm{Q}(\cdot)$,
by including departures seen as a death process.

The richness, and complexity, of the cluster representation is apparent from the cluster processes that each immigrant event generates.
For each immigrant $(T_r^{(0)},j)$ in component $j\in[d]$, we explicitly denote the $d$-dimensional cluster process it generates by $\bm{S}_j^{\bm{N}}(\cdot)$,
counting the events in each component of the cluster described in Part~(2) of Definition~\ref{def: hawkes cluster}.
Henceforth, the time $u= t- T_r^{(0)}$ is always understood as the remaining time
after the arrival of the immigrant event.
The cluster process $\bm{S}_j^{\bm{N}}(u)$ at time $u$ is then given by
\begin{align}
\label{eq: cluster N definition}
    \bm{S}_j^{\bm{N}}(u) :=
    \begin{bmatrix}
    S^{\bm{N}}_{1\leftarrow j}(u) \\
    \vdots \\
    S^{\bm{N}}_{d\leftarrow j}(u)
    \end{bmatrix},
\end{align}
where each entry $S^{\bm{N}}_{i\leftarrow j}(u)$ records the number of events generated in component $i$, up to and including time $u$,
with as oldest ancestor the immigrant event in component $j$ that generates the cluster $\bm{S}_j^{\bm{N}}(\cdot)$; see also 
\cite{BMM15}.
Each entry $S^{\bm{N}}_{i\leftarrow j}(\cdot)$ therefore contains part of the entire cluster described in Part~(2) of Definition~\ref{def: hawkes cluster}. 
(Note that the immigrant event $(T^{(0)},j)$ itself is included in the cluster when $i=j$ to avoid double counts.)
Due to our multivariate setting, events recorded in $S^{\bm{N}}_{i\leftarrow j}(\cdot)$ may have propagated through other dimensions $m\in[d]$ before arriving in $i$.
Thus, we have to meticulously keep track of the event times and components of descendant events.

Three basic distributional properties, which we will later exploit to operationalize the cluster representation, are noteworthy.
First, each immigrant event from source component $j\in[d]$ generates a cluster according to the same distribution
modulo a time shift to account for the arrival time;
see also \cite[Section 6.3]{DVJ07}.
Hence, we index the cluster process $\bm{S}_j^{\bm{N}}(\cdot)$ by the source component it corresponds to, not by individual immigrant.
Second, each event from the same source component generates offspring
according to the same iterative procedure and, as such, there is a \textit{branching structure}, and \textit{self-similarity}, underlying each cluster.
Third, the cluster processes are generated independently across source components:
$\bm{S}_j^{\bm{N}}(\cdot)\perp\!\!\!\perp \bm{S}_m^{\bm{N}}(\cdot)$, $j,m\in[d]$, $j\neq m$.


Our analysis will exploit the branching structure that underlies the cluster representation for all three processes $\bm{N}(\cdot)$, $\bm{Q}(\cdot)$ and $\bm{\lambda}(\cdot)$
to characterize distributional properties of the processes jointly.
To explicitly describe their dynamics, we introduce the following two cluster processes,
resembling the cluster process $\bm{S}_j^{\bm{N}}(\cdot)$ in \eqref{eq: cluster N definition}.
Denote the $\nn^d$-valued cluster process $\bm{S}_j^{\bm{Q}}(\cdot)$ and the $\rr_+^d$-valued cluster process $\bm{S}_j^{\bm{\lambda}}(\cdot)$,
corresponding to $\bm{Q}(\cdot)$ and $\bm{\lambda}(\cdot)$ respectively, at time $u$ by
\begin{align}
\label{eq: clusters Q,L definition}
    \bm{S}^{\bm{Q}}_j(u) :=
    \begin{bmatrix}
    S^{\bm{Q}}_{1\leftarrow j}(u) \\
    \vdots \\
    S^{\bm{Q}}_{d\leftarrow j}(u)
    \end{bmatrix},\:\:\:
    \
    \bm{S}^{\bm{\lambda}}_j(u) := \begin{bmatrix}
    S^{\bm{\lambda}}_{1\leftarrow j}(u) \\
    \vdots \\
    S^{\bm{\lambda}}_{d\leftarrow j}(u)
    \end{bmatrix}.
\end{align}
Here, $S^{\bm{Q}}_{i\leftarrow j}(u)$ equals $S^{\bm{N}}_{i\leftarrow j}(u)$ minus the departures in component $i$ up to and including time $u$.
Furthermore, $S^{\bm{\lambda}}_{i\leftarrow j}(u)$ records the aggregated change in the intensity of component $i$
caused by the jump sizes that are distributed as $B_{ij}$ and the decays $g_{ij}(\cdot)$ following events in the cluster $\bm{S}_j^{\bm{N}}(u)$ strictly prior to time $u$.

To fully appreciate these clusters and their entries,
we explore in the next section the distributional properties that occur as a consequence of the branching structure that underlies the different cluster processes.

\begin{remark}
\label{rem: focus on joint (Q,L) process}
Most of the results in this paper pertain to the joint process $(\bm{Q}(\cdot),\bm{\lambda}(\cdot))$,
which includes $(\bm{N}(\cdot),\bm{\lambda}(\cdot))$ as a special case when $J_{j}\equiv\infty$ for all $j\in[d]$.
We remark that it is possible to extend our results to cover the joint process $(\bm{N}(\cdot),\bm{Q}(\cdot),\bm{\lambda}(\cdot))$
at the cost of heavier notation.
\end{remark}

\section{Joint Transforms}\label{sec:JointTransforms}

In this section, by exploiting the branching structure underlying the cluster representation,
we first derive a collection of distributional equalities that play a key role
in next characterizing a general joint transform of the random object $(\bm{Q}(t),\bm{\lambda}(t))$, for any $t\in\rr$,
in terms of a semi-closed-form expression.
The generality of this characterization also allows us to obtain, as corollaries in specific cases,
several additional new transform results that are of independent interest.


\subsection{Analyzing distributional equalities}\label{sec:disequal}

The cluster processes appearing in the cluster representation describe
how the offspring events due to self- and cross-excitation are generated.
To illustrate how these cluster processes behave, Figure \ref{fig: bivariate sample path cluster processes} displays
a realization of $\bm{N}(\cdot)$ and $\bm{\lambda}(\cdot)$
with corresponding clusters $\bm{S}_j^{\bm{N}}(\cdot)$ and $\bm{S}_j^{\bm{\lambda}}(\cdot)$
in the bivariate case $d=2$.
In the upper two subplots, the dotted arrows between events (crosses and diamonds)
indicate how events are generated across time and components,
revealing the branching structure of the cluster processes.
In addition, we plot the cluster processes $S_{i\leftarrow j}^{\bm{N}}(\cdot)$ and $S_{i\leftarrow j}^{\bm{\lambda}}(\cdot)$, for $i,j=1,2$,
so as to make visible how they relate to the processes $N_i(\cdot)$ and $\lambda_i(\cdot)$.

\begin{figure}
\includegraphics[scale=.63]{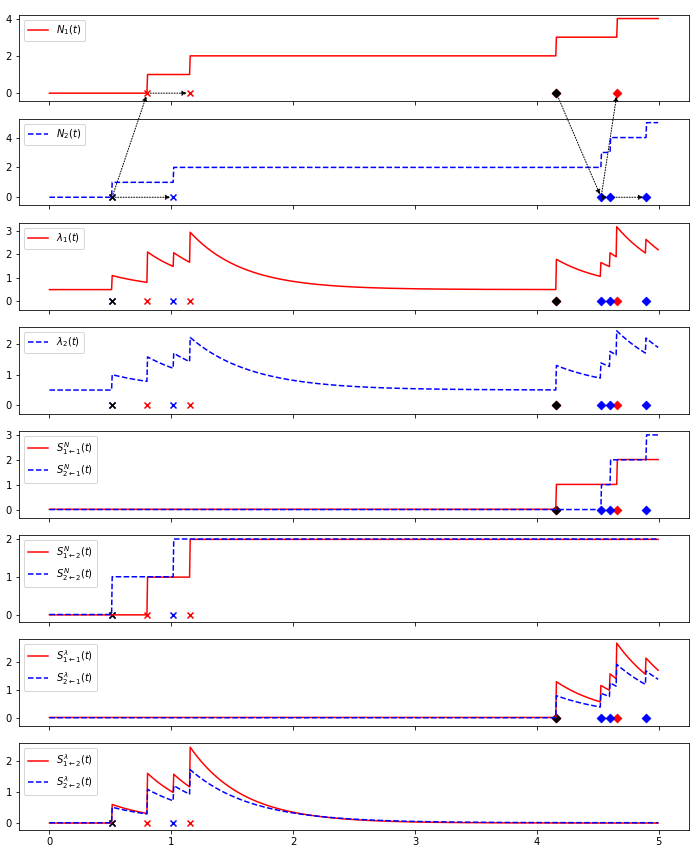} \\
\caption{\label{fig: bivariate sample path cluster processes}\small \textit{Sample paths of:
$\bm{N}(\cdot)=(N_1(\cdot),N_2(\cdot))$;
$\bm{\lambda}(\cdot)=(\lambda_1(\cdot),\lambda_2(\cdot))$;
the cluster processes originating in component 1 (diamonds),
$\bm{S}^{\bm{N}}_1(\cdot) = (S^{\bm{N}}_{1\leftarrow 1}(\cdot),S^{\bm{N}}_{2\leftarrow 1}(\cdot))$ and $\bm{S}^{\bm{\lambda}}_1(\cdot) = (S^{\bm{\lambda}}_{1\leftarrow 1}(\cdot), S^{\bm{\lambda}}_{2\leftarrow 1}(\cdot))$; and the cluster processes originating in component 2 (crosses),
$\bm{S}^{\bm{N}}_2(\cdot) = (S^{\bm{N}}_{1\leftarrow 2}(\cdot),S^{\bm{N}}_{2\leftarrow 2}(\cdot))$ and $\bm{S}^{\bm{\lambda}}_2(\cdot) = (S^{\bm{\lambda}}_{1\leftarrow 2}(\cdot), S^{\bm{\lambda}}_{2\leftarrow 2}(\cdot))$,
under exponential decay.
Parameters: $\lambda_{1,\infty} = \lambda_{2,\infty} = 0.5$, $\alpha_{11}=\alpha_{12} = 2.3$, $\alpha_{21} = \alpha_{22} = 2$, $B_{11} \equiv 1.3$, $B_{12} \equiv 0.6$, $B_{21} \equiv 0.8$, $B_{22} \equiv 0.5$.}}
\end{figure}


As the visualization in Figure \ref{fig: bivariate sample path cluster processes} suggests,
the clusters defined in \eqref{eq: cluster N definition} and \eqref{eq: clusters Q,L definition} are formally connected
to the processes $N_i(\cdot)$, $Q_i(\cdot)$ and $\lambda_i(\cdot)$ via a set of distributional equalities.
More precisely, the equivalence between the cluster representation in Definition~\ref{def: hawkes cluster}
and the intensity-based Definition~\ref{def: hawkes intensity}
allows us to probabilistically describe events, and their impact on $N_i(\cdot)$, $Q_i(\cdot)$ and $\lambda_i(\cdot)$,
in terms of the clusters of immigrants and offspring.
Indeed, for given $t\in\rr$,
\begin{align}
\label{eq: distributional equalities Q_i N_i lambda_i}
    \begin{split}
        N_i(t) &\overset{\rm d}{=} \sum_{j=1}^d \sum_{k = 0}^{I_j(t)} S^{\bm{N}}_{i\leftarrow j}(t-T_{k}),\\
        Q_i(t) &\overset{\rm d}{=} \sum_{j=1}^d \sum_{k = 0}^{I_j(t)} S^{\bm{Q}}_{i\leftarrow j}(t-T_{k}), \\
        \lambda_i(t) &\overset{\rm d}{=} \lambda_{i,\infty} + \sum_{j=1}^d \sum_{k = 0}^{I_j(t)} S^{\bm{\lambda}}_{i\leftarrow j}(t-T_{k}),
    \end{split}
\end{align}
where $(T_{k})_{k}$ are the immigrant event times
and $I_j(\cdot)$ is as in Definition \ref{def: hawkes cluster}.
The distributional equality concerning $\lambda_i(t)$ may be compared to Eqn.\ \eqref{eq: intensity lambda sum notation},
where we expressed $\lambda_i(t)$ as the pathwise aggregated change in intensity
due to all events strictly prior to time $t$.
We also briefly remark that we can express the relation between
the cluster entries of $\bm{S}_j^{\bm{N}}(\cdot)$ and $\bm{S}_j^{\bm{\lambda}}(\cdot)$,
at each time $u$,
by
\begin{align}
    S^{\bm{\lambda}}_{i\leftarrow j}(u) = \sum_{k=1}^d \sum_{r=1}^{S^{\bm{N}}_{k\leftarrow j}(u)} B_{ik,r}\,g_{ik}(u-T_r).
\end{align}

We now fix an immigrant event $(T^{(0)},j)$ in source component $j\in[d]$ and analyze the generated clusters $\bm{S}_j^\star(\cdot)$,
where $\star\in\{\bm{N},\bm{Q},\bm{\lambda}\}$.
To exploit the underlying branching structure,
and the self-similarity,
the idea consists in recognizing that this immigrant event generates first-generation events into all $m$ components,
and these in turn generate their own \textit{sub-clusters} $\bm{S}_m^\star(\cdot)$.
In order to formally capture this mechanism,
we define the matrix processes $\bm{S}^{\bm{N}}(\cdot)$, $\bm{S}^{\bm{Q}}(\cdot)$ and $\bm{S}^{\bm{\lambda}}(\cdot)$ by
\begin{align}
\label{eq: matrices S for N, Q, lambda def}
    \bm{S}^\star(\cdot) := \begin{bmatrix}
    \bm{S}_1^\star(\cdot) \ | \ \cdots \ | \ \bm{S}_d^\star(\cdot)
    \end{bmatrix}
    = \begin{bmatrix}
    S^\star_{1\leftarrow 1}(\cdot) & \cdots & S^\star_{1 \leftarrow d}(\cdot)\\
    \vdots & \ddots & \vdots\\
    S^\star_{d\leftarrow 1}(\cdot)& \cdots & S^\star_{d\leftarrow d}(\cdot)
    \end{bmatrix},
\end{align}
for $\star \in \{\bm{N},\bm{Q},\bm{\lambda}\}$.
The columns $\bm{S}_j^\star(\cdot)$ correspond to the clusters defined in \eqref{eq: cluster N definition} and \eqref{eq: clusters Q,L definition}
and keep track of offspring events that \textit{originate} from component $j$,
while the rows, in the sequel denoted by $\bm{S}_{(i)}^\star(\cdot)$,
record offspring events that \textit{arrive into} component $i$.
Observe that the right-hand side expressions in \eqref{eq: distributional equalities Q_i N_i lambda_i} contain precisely the entries of the rows.

From Section~\ref{sec:ModelBranching} we know that the underlying branching structure is similar
for the clusters corresponding to $\bm{N}(\cdot),\bm{Q}(\cdot)$ and $\bm{\lambda}(\cdot)$.
This fact leads us to introduce unifying notation.
To this end, we define the functional $\cA_j$, $j\in[d]$, that acts on $\bm{X}(\cdot)=(X_1(\cdot),\ldots,X_d(\cdot))$,
a (row-)vector-valued process taking values in $\rr_+^d$,
and $P\geqslant 0$, for each time $u$ by
\begin{align}
\label{eq: def linear functional A_ij branching structure first gen}
    \cA_{j}\big\{P,\bm{X}(\cdot)\big\}(u) = P + \sum_{m=1}^d \sum_{k=1}^{K_{mj}(u)} X_m(u-T_{k}),
\end{align}
where $P$ accounts for the impact of the immigrant event and the terms in the summations account for the impact of offspring events,
with $K_{mj}(\cdot)$ as in Definition~\ref{def: hawkes cluster}.
Note the time shift to account for the arrival time $T_k$ of the offspring event.

The functional $\cA_{j}$ allows us to compactly and coherently formulate distributional equalities for the respective cluster processes.
Indeed, zooming in on specific components $S_{i\leftarrow j}^\star(\cdot)$, with $\star \in \{\bm{N},\bm{Q},\bm{\lambda}\}$,
yields the micro-level distributional equalities at time $u$ given by
\begin{align}
\label{eq: distributional equality S_ij}
\begin{split}
    S^{\bm{N}}_{i \leftarrow j}(u) &\overset{\rm d}{=} \cA_j\big\{\bm{1}_{\{i=j\}}, \bm{S}_{(i)}^{\bm{N}}(\cdot)\big\}(u),\\
    S^{\bm{Q}}_{i \leftarrow j}(u) &\overset{\rm d}{=} \cA_j\big\{\bm{1}_{\{i=j\}}\bm{1}_{\{J_i>\,u\,\}}, \bm{S}_{(i)}^{\bm{Q}}(\cdot)\big\}(u),\\
    S^{\bm{\lambda}}_{i\leftarrow j}(u) &\overset{\rm d}{=} \cA_j\big\{ B_{ij}\,g_{ij}(u), \bm{S}_{(i)}^{\bm{\lambda}}(\cdot)\big\}(u),
\end{split}
\end{align}
which will prove to play a crucial role in the analysis of the cluster processes that follows (in Section~\ref{sec:FixedPoint}).
Note the difference in the first arguments for the different processes,
and note that the second argument $\bm{S}_{(i)}^\star(\cdot) = (S_{i\leftarrow 1}^\star(\cdot),\dots,S_{i\leftarrow d}^\star(\cdot))$
is the $i$-th row of the matrix $\bm{S}^\star(\cdot)$, which accounts for the offspring events.
The $B_{ij}$'s in the expression for $S_{i\leftarrow j}^{\bm{\lambda}}(\cdot)$ are understood to be sampled for each event in the cluster.
Intuitively, Eqn.\ \eqref{eq: distributional equality S_ij} says that the total impact of a cluster process
from source component $j$ on target component $i$
is equal in distribution to the superposition of first-generation events and the impact of their offspring.
The equality for $S_{i\leftarrow j}^{\bm{Q}}(\cdot)$ in \eqref{eq: distributional equality S_ij} is the multivariate counterpart of \cite[Eqn.\ (4.20)]{KSBM18},
in the sense that they coincide when setting $d=1$ in our setup.

\begin{remark}
The distributional equalities in \eqref{eq: distributional equality S_ij} can be extended
to the vectors $\bm{S}_j^\star(\cdot)$ for $\star\in\{\bm{N},\bm{Q},\bm{\lambda}\}$ using the mapping $\bm{\cA}_j$,
defined at time $u$ by
\begin{align}
\label{eq: def linear functional branching structure first gen}
    \bm{\cA}_j\big\{\bm{P},\bm{X}(\cdot)\}(u) = \bm{P} + \sum_{m=1}^d \sum_{k=1}^{K_{mj}(u)} \bm{X}_m(u-T_{k}).
\end{align}
Here, $\bm{P}\in\rr_+^d$ accounts for the immigrant event, $\bm{X}(\cdot)$ is an $\rr_+^{d\times d}$-valued matrix process,
and $\bm{X}_m(\cdot)$ is its $m$-th column vector.
For $\star\in\{\bm{N},\bm{Q},\bm{\lambda}\}$, one can substitute appropriate values $\bm{P}^\star(u)$ for $\bm{P}$
and use the matrix $\bm{S}^\star(\cdot)$ defined in \eqref{eq: matrices S for N, Q, lambda def}
to account for the offspring events,
to obtain the vector-valued versions of Eqn.\ \eqref{eq: distributional equality S_ij}.
Note that \eqref{eq: def linear functional branching structure first gen} describes the underlying branching structure of entire clusters $\bm{S}_j^\star(\cdot)$,
and that the entries of $\bm{\cA}_j$ correspond to $\cA_{1},\dots,\cA_d$.
\end{remark}

\subsection{Transform characterization}\label{sec:transformcha}

The distributional equalities are key to characterize a general joint transform of $(\bm{Q}(t),\bm{\lambda}(t))$.
We first make precise what we mean by joint transform.

\begin{definition}
\label{def: joint transform vectors}
Let $(\bm{X}(\cdot),\bm{Y}(\cdot))$ be a stochastic process taking values in $\nn_+^d \times \rr_+^d$.
For any $t\in\rr_+$, the joint transform of $(\bm{X}(t),\bm{Y}(t))$ is given by
\begin{align}
\label{eq: general def joint transform}
    \cJ_{\bm{X},\bm{Y}}(t)&\equiv
    \cJ_{\bm{X},\bm{Y}}(t,\bm{s},\bm{z}):= \ee\Big[  \bm{z}^{\bm{X}(t)}  e^{-\bm{s}^\top \bm{Y}(t)} \Big] \equiv
    \ee\Big[\prod_{i=1}^d  z_i^{X_i(t)}e^{-s_i Y_i(t)} \Big],
\end{align}
where $\bm{s}\in \rr_+^d$ and $\bm{z}\in[-1,1]^d$,
and we denote the space of such transforms by $\jj$, such that $\cJ_{\bm{X},\bm{Y}}(\cdot) \in \jj$.
Here, $\mathbb{E}\left[\cdot\right]$ is understood as $\mathbb{E}_{0}\left[\cdot\right]$,
i.e., expectation conditional upon the respective filtration at $t=0$.
\end{definition}
Note that Eqn.~\eqref{eq: general def joint transform} is well-defined, i.e., exists for any $t\in\rr_+$, $\bm{s}\in \rr_+^d$, and $\bm{z}\in[-1,1]^d$.
Throughout the paper, $\bm{s}\in \rr_+^d$ and $\bm{z}\in[-1,1]^d$ remain fixed, unless stated otherwise,
and are therefore sometimes suppressed in the notation for readability.
In our setting, we consider the joint transform of $(\bm{Q}(t),\bm{\lambda}(t))$,
with $\bm{Q}(0)=\bm{0}$ and $\bm{\lambda}(0)=\bm{\lambda}_{\infty}$,
given by
\begin{align}
\label{eq: joint transform Q,lambda}
    \cJ_{\bm{Q},\bm{\lambda}}(t) =  \ee\Big[\prod_{i=1}^d  z_i^{Q_i(t)} e^{-s_i \lambda_i(t)}\Big].
\end{align}

We proceed to show that we can obtain a semi-closed-form expression for $\cJ_{\bm{Q},\bm{\lambda}}(t)$
using the distributional properties derived in Section~\ref{sec:disequal}.
Specifically, we use Eqn.~\eqref{eq: distributional equalities Q_i N_i lambda_i}
to describe the entries $Q_i(\cdot)$ and $\lambda_i(\cdot)$ in terms of the respective cluster processes $\bm{S}_j^{\bm{Q}}(\cdot)$ and $\bm{S}_j^{\bm{\lambda}}(\cdot)$.
To that end, we also need to consider the joint transform of $(\bm{S}_j^{\bm{Q}}(u),\bm{S}_j^{\bm{\lambda}}(u))$,
with $\bm{S}_j^{\bm{Q}}(0) = \bm{e}_j$, the unit vector with $j$-th entry equal to $1$, and $\bm{S}_j^{\bm{\lambda}}(0) = \bm{B}^j$, given by
\begin{align}
\label{eq: joint transform clusters Q,lambda}
    \cJ_{\bm{S}_j^{\bm{Q}},\bm{S}_j^{\bm{\lambda}}}(u) =  \ee\Big[\prod_{i=1}^d z_i^{S_{i \leftarrow j}^{\bm{Q}}(u)}e^{-s_i S_{i \leftarrow j}^{\bm{\lambda}}(u)}\Big].
\end{align}

We can now state the first main result regarding the joint transform $\cJ_{\bm{Q},\bm{\lambda}}(t)$,
expressed in terms of $\cJ_{\bm{S}_j^{\bm{Q}},\bm{S}_j^{\bm{\lambda}}}(u)$ for $j\in[d]$ and $u\in[0,t]$;
later in this paper (in Section~\ref{sec:FixedPoint}) it is shown how the $\cJ_{\bm{S}_j^{\bm{Q}},\bm{S}_j^{\bm{\lambda}}}(u)$ can be identified.
By using the cluster representation, the independence of the cluster processes across components,
and exploiting the derived distributional equalities, we establish the following identity.

\begin{theorem}
\label{thm: joint transform zeta characterization}
The joint transform $\cJ_{\bm{Q},\bm{\lambda}}(t)$ satisfies
\begin{align}
\label{eq: joint transform lambda,queue}
\begin{split}
    \cJ_{\bm{Q},\bm{\lambda}}(t,\bm{s},\bm{z})
    &= \prod_{j=1}^d \exp\Big(-\lambda_{j,\infty}(t+s_j) + \lambda_{j,\infty}\int_0^t \cJ_{\bm{S}_j^{\bm{Q}},\bm{S}_j^{\bm{\lambda}}}(u,\bm{s},\bm{z})\mathrm{d}u\Big).
\end{split}
\end{align}
\end{theorem}

\begin{proof}
We start by conditioning on the number of immigrants in each component, and use the fact that these arrive independently.
For brevity, we introduce the vectors $\bm{I}(t) = (I_1(t),\dots,I_d(t))^\top$ of immigrant processes
and $\bm{n} = (n_1,\dots,n_d)^\top  \in \nn_0^d$ (with $\nn_0:=\{0,1,2,\ldots\}$).
From the distributional equalities \eqref{eq: distributional equalities Q_i N_i lambda_i}, we obtain
\begin{align*}
    \cJ_{\bm{Q},\bm{\lambda}}(t,\bm{s},\bm{z})
    &= \sum_{\bm{n}\in {\mathbb N}_0^d} \ee\Big[ e^{-\bm{s}^\top \bm{\lambda}(t)} \bm{z}^{\bm{Q}(t)} \, \Big| \, \bm{I}(t) = \bm{n} \Big] \,\pp(\bm{I}(t) = \bm{n}) \\
    &= \sum_{\bm{n}\in {\mathbb N}_0^d}  \ee\Big[\prod_{j=1}^d e^{-s_j\lambda_{j,\infty}}\prod_{i=1}^d e^{- s_i \sum_{k =1}^{n_j} S^{\bm{\lambda}}_{i\leftarrow j}(t-T_{k})} z_i^{\sum_{k=1}^{n_j}S^{\bm{Q}}_{i\leftarrow j}(t-T_{k})} \Big]\, \pp(\bm{I}(t) = \bm{n}) \\
    &=\sum_{\bm{n}\in {\mathbb N}_0^d} \prod_{j=1}^d e^{-s_j\lambda_{j,\infty}} \ee\Big[  \prod_{i=1}^d e^{- s_i \sum_{k =1}^{n_j} S^{\bm{\lambda}}_{i\leftarrow j}(t-T_{k})} z_i^{\sum_{k=1}^{n_j} S^{\bm{Q}}_{i\leftarrow j}(t-T_{k})}\Big]\, \pp(\bm{I}(t) = \bm{n}),
\end{align*}
where we have used the independence between the clusters and immigrant processes in the second equality,
the independence among clusters in the last equality,
and write
\[\pp(\bm{I}(t) = \bm{n}) = \prod_{j=1}^d \pp(I_j(t)=n_j),\]
for brevity.
Recalling that each $I_j(t)$ is a Poisson process, we can use the property that conditional on the number of events at time $t$,
the event arrival times are i.i.d.\ according to a uniformly distributed random variable on the interval $[0,t]$.
With $T^{(j)}$ being uniformly distributed on $[0,t]$ (independent of anything else, that is), we thus have that,
for each $j\in[d]$, the sequence $(T_{k})_{k\in[n_j]}$ are i.i.d.\ copies of $T^{(j)}$.
This allows us to write
\begin{align*}
    \ee\Big[ \prod_{i=1}^d e^{- s_i \sum_{k =1}^{n_j} S^{\bm{\lambda}}_{i\leftarrow j}(t-T_{k})} z_i^{\sum_{k=1}^{n_j} S^{\bm{Q}}_{i\leftarrow j}(t-T_{k})}\Big]
    &=\Big(\ee\Big[ \prod_{i=1}^d e^{- s_i S^{\bm{\lambda}}_{i\leftarrow j}(t-T^{(j)})} z_i^{S^{\bm{Q}}_{i\leftarrow j}(t-T^{(j)})}\Big]\Big)^{n_j}
    \\&= \big(\cJ_{\bm{S}_j^{\bm{Q}},\bm{S}_j^{\bm{\lambda}}}(u-T^{(j)})\big)^{n_j},
\end{align*}
by the definition of $\cJ_{\bm{S}_j^{\bm{Q}},\bm{S}_j^{\bm{\lambda}}}$.
Now using that $T^{(j)}$ is uniformly distributed on $[0,t]$ and that $I_j(\cdot)$ are Poisson processes with rate $\lambda_{j,\infty}$,
we obtain that
\begin{align*}
   \cJ_{\bm{Q},\bm{\lambda}}(t)  &=\sum_{\bm{n}\in {\mathbb N}_0^d} \prod_{j=1}^d e^{-s_j\lambda_{j,\infty}}\Big( \frac{1}{t}\int_0^t \cJ_{\bm{S}_j^{\bm{Q}},\bm{S}_j^{\bm{\lambda}}}(t-u)\mathrm{d}u\Big)^{n_j} \frac{(\lambda_{j,\infty}t)^{n_j}}{n_j!}e^{-\lambda_{j,\infty}t} \\
    &=\sum_{\bm{n}\in {\mathbb N}_0^d} \prod_{j=1}^d e^{-\lambda_{j,\infty}(t+s_j)}\frac{\Big(\lambda_{j,\infty} \int_0^t \cJ_{\bm{S}_j^{\bm{Q}},\bm{S}_j^{\bm{\lambda}}}(u)\mathrm{d}u\Big)^{n_j}}{n_j!}\\
    &=\prod_{j=1}^d \exp\Big(-\lambda_{j,\infty}(t+s_j) + \lambda_{j,\infty}\int_0^t \cJ_{\bm{S}_j^{\bm{Q}},\bm{S}_j^{\bm{\lambda}}}(u))\mathrm{d}u\Big),
\end{align*}
where the second equality holds by an elementary change of variables.
\end{proof}

We proceed by presenting a number of corollaries of Theorem~\ref{thm: joint transform zeta characterization},
which are already new to the literature in their own right.
The first one considers the univariate case.

\begin{corollary}
\label{cor: joint transform univariate}
In case $d=1$, the joint transform $\cJ_{Q,\lambda}(t)$ of $(Q(t),\lambda(t))$ is given by
\begin{align}
\label{eq: corollary joint transform univariate}
    \ee\Big[ z^{Q(t)}e^{-s\lambda(t)}\Big]
    &= \exp\Big(-\lambda_{\infty}(t+s) + \lambda_{\infty}\int_0^t \cJ_{S^{\bm{Q}},S^{\bm{\lambda}}}(u,s,z)\,\mathrm{d}u\Big),
\end{align}
with $s\in\rr_+$ and $z\in[-1,1]$ and
where $\cJ_{S^{\bm{Q}},S^{\bm{\lambda}}}(u,s,z)$ is defined in the obvious manner.
\end{corollary}
By taking $s=0$ in Corollary~\ref{cor: joint transform univariate} one recovers the probability generating function of $Q(t)$
that was established in \cite{KSBM18}. 
The second corollary considers $\bm{Q}(t)$ and $\bm{\lambda}(t)$ separately.

\begin{corollary}
\label{cor: joint transform multivariate}
By taking $s_i = 0$ for all $i\in[d]$ in \eqref{eq: joint transform lambda,queue},
we obtain the probability generating function of $\bm{Q}(t)$:
\begin{align}
\label{eq: corollary pgf queue}
    \ee\Big[ \bm{z}^{\bm{Q}(t)}\Big]
    &= \prod_{j=1}^d \exp\Big( -\lambda_{j,\infty}t + \lambda_{j,\infty}\int_0^t \ee\Big[ \prod_{i=1}^d z_i^{S^{\bm{Q}}_{i\leftarrow j}(u)}\Big] \mathrm{d}u\Big).
\end{align}
By taking $z_i =1$ for all $i\in[d]$ in \eqref{eq: joint transform lambda,queue},
we obtain the Laplace-Stieltjes transform of $\bm{\lambda}(t)$:
\begin{align}
\label{eq: corollary Laplace-Stieltjes transform lambda}
    \ee\Big[ e^{-\bm{s}^\top \bm{\lambda}(t)}\Big]
    &= \prod_{j=1}^d \exp\Big( -\lambda_{j,\infty}(t+s_j) + \lambda_{j,\infty}
    \int_0^t \ee\Big[\prod_{i=1}^d e^{- s_i S^{\bm{\lambda}}_{i\leftarrow j}(u)}\Big] \mathrm{d}u\Big).
\end{align}
\end{corollary}

As noticed before, choosing $J_j\equiv \infty$ for all $j\in[d]$ implies that $\bm{Q}(\cdot) = \bm{N}(\cdot)$.
The following corollary is based on that observation.
\begin{corollary}
The joint transform $\cJ_{\bm{N},\bm{\lambda}}(t)$ of $(\bm{N}(t),\bm{\lambda}(t))$ is given by
\begin{align}
\label{eq: corollary joint transform Hawkes}
    \cJ_{\bm{N},\bm{\lambda}}(t)
    &=\prod_{j=1}^d \exp\Big(-\lambda_{j,\infty}(t+s_j) + \lambda_{j,\infty}\int_0^t \cJ_{\bm{S}_j^{\bm{N}},\bm{S}_j^{\bm{\lambda}}}(u,\bm{s},\bm{z})\,\mathrm{d}u\Big).
    \end{align}
\end{corollary}

In multivariate time-series analysis, where multivariate point processes play an important role, one often wants to compute auto- and cross-covariances,
involving expressions of the form $\ee[Q_i(t)Q_j(t+\tau)]$ for a combination of $i,j\in[d]$ and where $\tau>0$,
which in turn enable one to compute the respective auto- and cross-correlation functions.
Indeed, these functions are central objects in the identification and statistical inference of multivariate Hawkes processes; see e.g., \cite{ACL15} in the Markovian case.
In the next proposition, we provide a characterization of the probability generating function of $\bm{Q}(\cdot)$ associated with different time points, by extending Theorem~\ref{thm: joint transform zeta characterization}.
We note that this characterization may be extended further to include $\bm{\lambda}(\cdot)$ as well as to cover any finite number of time points.
\begin{proposition}
For $\bm{y},\bm{z} \in [-1,1]^d$ and $\tau > 0$, we have that
\begin{align}
\begin{split}
    \ee\Big[\prod_{i=1}^d y_i^{Q_i(t)} z_i^{Q_i(t+\tau)}\Big]
    = &\prod_{j=1}^d \exp\Big( \lambda_{j,\infty}\int_0^t \big(\ee\Big[ \prod_{i=1}^d (y_iz_i)^{S_{i\leftarrow j}^{\bm{Q}}(u)}\Big] -1\big)\mathrm{d}u\Big) \\
    &\times \exp\Big( \lambda_{j,\infty}\int_t^{t+\tau} \big(\ee\Big[ \prod_{i=1}^d z_i^{S_{i\leftarrow j}^{\bm{Q}}(u)}\Big] -1\big)\mathrm{d}u\Big).
\end{split}
\end{align}
\end{proposition}

\begin{proof}
The proof is similar to the proof of Theorem~\ref{thm: joint transform zeta characterization}.
We therefore omit the details and only explain the general structure of the proof.
Compared to Theorem~\ref{thm: joint transform zeta characterization}, one now has to condition twice: first one conditions on the number of immigrant events up to time $t$, and next, given the information up to time $t$, one re-conditions on the number of immigrant events up to time $t+\tau$.
By the independent increments property of the (immigrant) Poisson processes, the immigrant event arrival times for the respective conditioning events are uniformly distributed among the intervals $[0,t]$ and $[t,t+\tau]$.
Properly using the independence among the clusters similar to Theorem~\ref{thm: joint transform zeta characterization} and collecting terms, then yields the stated result.
\end{proof}

\begin{remark}
A related process, with ample applications in e.g.,\ insurance and risk,
is the multivariate compound Hawkes process, constructed as follows.
For each $i\in[d]$, let $(U_i^{(n)})_{n\in\nn}$ be a sequence of non-negative i.i.d.\ random variables independent of $\bm{N}(\cdot)$.
Define the multivariate compound Hawkes process $\bm{Z}(\cdot) := (Z_1(\cdot), \ldots, Z_d(\cdot))^{\top}$ entry-wise for $t\in\rr$ by
\begin{align}
\label{eq: compound hawkes entry}
    Z_i(t) := \sum_{n=1}^{N_i(t)} U_i^{(n)}.
\end{align}
For fixed $t\in\rr$, the Laplace-Stieltjes transform $\cT\{\bm{Z}(t)\}(\bm{s})$ of $\bm{Z}(t)$ satisfies
\begin{align}
\label{eq: compound hawkes characterization}
    \cT\{\bm{Z}(t)\}(\bm{s}) = \ee\big[e^{-\bm{s}^\top \bm{Z}(t)}\big]
    = \ee\Big[ \prod_{i=1}^d \big(\cT\{U_i\}(s_i)\big)^{N_i(t)} \Big],
\end{align}
where $\cT\{U_i\}(s) = \ee[e^{-sU_i}]$ is the Laplace-Stieltjes transform of $U_i$ evaluated in $s$.
Also observe that $\cT\{\bm{Z}(t)\}(\bm{s})$ can be expressed in terms of quantities discussed earlier in this section,
as the right-hand side of \eqref{eq: compound hawkes characterization} can be interpreted
as the probability generating function of $\bm{N}(t)$
evaluated in $\bm{z}=\cT\{\bm{U}\}(\bm{s})\equiv (\cT\{U_1\}(s_1),\ldots,\cT\{U_d\}(s_d))$.
In other words,
\begin{align}
    \cT\{\bm{Z}(t)\}(\bm{s})  = \cJ_{\bm{N},\bm{\lambda}}\big(t, \bm{0}, \cT\{\bm{U}\}(\bm{s})\big).
\end{align}
\end{remark}


\section{Fixed-Point Theorem}\label{sec:FixedPoint}

In the previous section, we expressed the joint transform $\cJ_{\bm{Q},\bm{\lambda}}(t)$  of $(\bm{Q}(t),\bm{\lambda}(t))$ in terms of
\begin{align}
\label{eq:J_cluster}
{\mathcal G}_j(u):=\cJ_{\bm{S}_j^{\bm{Q}},\bm{S}_j^{\bm{\lambda}}}(u)\equiv \cJ_{\bm{S}_j^{\bm{Q}},\bm{S}_j^{\bm{\lambda}}}(u,\bm{s},\bm{z});
\end{align}
see in particular Eqn.\ \eqref{eq: joint transform lambda,queue} in the characterization of the joint transform $\cJ_{\bm{Q},\bm{\lambda}}(t)$
that is given in Theorem \ref{thm: joint transform zeta characterization}.
In this section, we focus our analysis on \eqref{eq:J_cluster}.
More specifically, by employing the previously derived distributional equalities,
we characterize the joint transform \eqref{eq:J_cluster}
in terms of the fixed point of a certain mapping.
We also provide an explicit iteration scheme that, as we formally prove, converges to this fixed point.

\subsection{Spaces of joint transforms}

Recall from Definition \ref{def: joint transform vectors} the space $\jj$ of time-dependent joint transforms of $d$-dimensional vector-valued processes $(\bm{X}(\cdot),\bm{Y}(\cdot))$.
To handle the matrices $\bm{S}^\star(\cdot)$ for $\star \in \{\bm{N},\bm{Q},\bm{\lambda}\}$, we extend the space $\jj$ to include matrices, as follows.

\begin{definition}
\label{def: time dependent joint transform space}
Set $\jj^d$ to be the $d$-dimensional analogue of $\jj$, in the sense that an element $\bm{\cJ}_{\bm{X},\bm{Y}}(\cdot) \in \jj^d$ is given at time $u$ by
\begin{align}
\label{eq: def joint transform d-dimensional}
    \bm{\cJ}_{\bm{X},\bm{Y}}(u)
    := \begin{bmatrix}
    \cJ_{\bm{X}_1,\bm{Y}_1}(u) \\
    \vdots \\
    \cJ_{\bm{X}_d,\bm{Y}_d}(u)
    \end{bmatrix},
\end{align}
where for each $j\in[d]$, the entry $\cJ_{\bm{X}_j,\bm{Y}_j}(\cdot)\in \jj$ is the joint transform
corresponding to  $(\bm{X}_j(u),\bm{Y}_j(u)) = ((X_{1j}(u),Y_{1j}(u)),\dots,(X_{dj}(u),Y_{dj}(u)))^\top$
as defined in Eqn.\ \eqref{eq: general def joint transform}.
\end{definition}
Note that an entry on the right-hand side of Eqn.\ (\ref{eq: def joint transform d-dimensional}) can be viewed as the joint transform of the columns of the
matrix-valued random object $(\bm{X}(u),\bm{Y}(u)) = (X_{ij}(u),Y_{ij}(u))_{i,j\in[d]}$.

When considering the processes $(\bm{S}^{\bm{Q}}(\cdot),\bm{S}^{\bm{\lambda}}(\cdot))$,
recall that for each $j\in[d]$ we have that the transform ${\mathcal G}_j(\cdot)$ is an element of $\jj$ by Eqn.\ (\ref{eq: joint transform clusters Q,lambda}).
The space $\jj^d$ plays an important role in the exploitation of the distributional equalities given in Eqn.\ (\ref{eq: distributional equality S_ij}).
Since in our general multivariate Hawkes model a cluster originating in source component $j$ can in principle
generate events in any of the components,
characterizing ${\mathcal G}_j(\cdot)$
requires us to \textit{simultaneously} consider
${\mathcal G}_m(\cdot)$
for {\it all}\, $m\in[d]$.
This explains why we work with the following vector of time-dependent joint transforms:
\begin{align}
\label{eq: joint transform vector S^Q,S^lambda}
{\boldsymbol {\mathcal G}}(u):=
    \bm{\cJ}_{\bm{S}^{\bm{Q}},\bm{S}^{\bm{\lambda}}}(u)
    &=
    \begin{bmatrix}
    \cJ_{\bm{S}_1^{\bm{Q}},\bm{S}_1^{\bm{\lambda}}}(u) \\
    \vdots \\
    \cJ_{\bm{S}_d^{\bm{Q}},\bm{S}_d^{\bm{\lambda}}}(u)
    \end{bmatrix}
    =
    \begin{bmatrix}
    {\mathcal G}_1(u) \\
    \vdots \\
    {\mathcal G}_d(u)
    \end{bmatrix}.
\end{align}
Here, as for any $j\in[d]$ the entry ${\mathcal G}_j(\cdot)$ is in
$\jj$ by Eqn.\ (\ref{eq: joint transform clusters Q,lambda}), we have that ${\boldsymbol {\mathcal G}}(\cdot) \in \jj^d$.
In words, $\bm{\cG}(u)$ is the vector containing the time-dependent joint transforms
corresponding to all pairs of cluster processes $(\bm{S}^{\bm{Q}}_j(u),\bm{S}^{\bm{\lambda}}_j(u))$ for $j\in[d]$,
i.e., the columns of $(\bm{S}^{\bm{Q}}(u),\bm{S}^{\bm{\lambda}}(u))$.
This informally entails that the object ${\boldsymbol {\mathcal G}}(u)$ contains a full probabilistic description of all underlying components.

Next, we state the following definition.
\begin{definition}
\label{def: phi definition}
Consider the mapping $\phi: \jj^d \to \jj^d$ which maps an element $\bm{\cJ} \equiv \bm{\cJ}_{\bm{X},\bm{Y}}(\cdot)  \in \jj^d$ to:
\begin{align}
\label{eq: phi vector definition}
    \bm{\cJ}(\cdot)=
    \begin{bmatrix}
    \cJ_1(\cdot)\\
    \vdots\\
    \cJ_d(\cdot)
    \end{bmatrix}
    \mapsto
    \begin{bmatrix}
    \phi_1(\cJ_1,\dots,\cJ_d)(\cdot)\\
    \vdots \\
    \phi_d(\cJ_1,\dots,\cJ_d)(\cdot)
    \end{bmatrix}
    =
    \begin{bmatrix}
    \phi_1(\bm{\cJ})(\cdot)\\
    \vdots \\
    \phi_d(\bm{\cJ})(\cdot)
    \end{bmatrix}
    = \phi(\bm{\cJ})(\cdot),
\end{align}
where each entry $\phi_j(\bm{\cJ})(\cdot)\in \jj$ is defined at time $u$ by
\begin{align}
\label{eq: phi_j definition mapping}
    \phi_j(\bm{\cJ})(u) &\equiv \phi_j(\cJ_1,\dots,\cJ_d)(u,\bm{s},\bm{z})\notag  \\
    \quad
    &:=\ee\Big[z_j^{\bm{1}_{\{J_j>u\}}}\Big] \prod_{i=1}^d\ee\Big[ \exp\big(-s_iB_{ij}g_{ij}(u)\big)\Big]  \\
    &\quad \times\prod_{m=1}^d \ee\Big[ \exp\Big( -B_{mj} \int_0^u g_{mj}(v)\big(1 - \cJ_m(u-v)\big)\,\mathrm{d}v\Big)\Big].\notag
\end{align}
\end{definition}

Note that, in Definition \ref{def: phi definition},
we suppressed the function notation $\bm{\cJ} \equiv \bm{\cJ}(\cdot)$ in the argument of $\phi$ for ease of readability,
allowing us to denote the image of $\bm{\cJ}(\cdot)$ as $\phi(\bm{\cJ})(\cdot)$.
It is not immediately clear that the mapping $\phi$ is well-defined, i.e.,
that for any $\bm{\cJ}_{\bm{X},\bm{Y}}(\cdot)\in\jj^d$, we have $\phi(\bm{\cJ}_{\bm{X},\bm{Y}})(\cdot) \in \jj^d$ as well.
We can show that $\phi(\bm{\cJ}_{\bm{X},\bm{Y}})(\cdot) \in \jj^d$, as desired, by suitably modifying
the arguments used e.g., in the proof of \cite[Theorem 1]{AW92} to this more complex setting.

\begin{lemma}
\label{lem: phi well defined}
The mapping $\phi$ in Eqn.\ \eqref{eq: phi vector definition} is well-defined.
\end{lemma}
\begin{proof}
See Section \ref{section: proofs} in the Appendix.
\end{proof}

\subsection{Fixed point and convergence results}

In this subsection, we characterize  $\bm{\cJ}_{\bm{S}^{\bm{Q}},\bm{S}^{\bm{\lambda}}}(u)$ in terms of a fixed point involving the mapping $\phi$.
We then show that iterating the mapping $\phi$ leads us to a unique limit
(i.e., the value of $\bm{\cJ}_{\bm{S}^{\bm{Q}},\bm{S}^{\bm{\lambda}}}(u)$ that we are after).  
To facilitate the analysis, we need to define an appropriate notion of distance for the space $\jj^d$.
We endow the space $\jj^d$ with the topology induced by the norm $\jjdnorm{\cdot}$, defined by
\begin{align*}
    \jjdnorm{\bm{\cJ}} := \sup_{\substack{u\in[0,t] \\ \bm{s}\in\rr_+^d \\ \bm{z}\in[-1,1]^d}} \rdnorm{\bm{\cJ}(u,\bm{s},\bm{z})} \equiv \sup_{u,\bm{s},\bm{z}} \rdnorm{\bm{\cJ}(u,\bm{s},\bm{z})}.
\end{align*}
The following result can be proven by suitably applying standard topological methods.

\begin{lemma}
\label{lem: phi continuous and bounds}
The mapping $\phi$ in Eqn.\ \eqref{eq: phi vector definition} is continuous with respect to the norm $\jjdnorm{\cdot}$.
\end{lemma}
\begin{proof}
See Section \ref{section: proofs} in the Appendix.
\end{proof}

Before we can state the main results of this section, we need an intermediate result.
In Definition \ref{def: hawkes cluster}, we saw that every event in source component $j$ generates events into target component $m$ according to an inhomogeneous Poisson process $K_{mj}(\cdot)$, with intensity $B_{mj}g_{mj}(\cdot)$, with $B_{mj}$ understood to be sampled at every event in $N_j(\cdot)$.
We need to specify when the offspring events arrive exactly, since these can generate further offspring only after they arrive.
Given $u$ as the remaining time after the arrival of the source event, let $v\leqslant u$ and denote by $H_{ij}(v\,|\,u)$ the probability that an offspring event was already generated before $v$, conditional on it being generated before $u$.
Also recall that each first-generation event generates a sub-cluster, as part of the original cluster.

\begin{lemma}
\label{lem: cluster independence children}
Consider the cluster process $\bm{S}_j^\star(\cdot)$ for $\star\in\{\bm{N},\bm{Q},\bm{\lambda}\}$ generated by an immigrant event $(T^{(0)},j)$ in component $j\in[d]$ and let $u=t-T^{(0)}$ be the time after this arrival.
Then the following statements hold:
\begin{enumerate}[(i)]
    \item Sub-clusters are i.i.d.; more precisely, for each $m\in[d]$, the sequence
    \begin{align*}
        \Big(\bm{S}_m^\star(u-T_k)\Big)_{k\in[n]},
    \end{align*} is an i.i.d.\ sequence, conditional on $\{ K_{mj}(u) = n\}$ for some $n\in\nn$ and with $(T_k)_k$ the arrival times of the first-generation events. 
    \item For $v \leqslant u$, the density $h_{ij}(v) = \frac{\mathrm{d}}{\mathrm{d}v}H_{ij}(v\,|\,u)$ is given by
    \begin{align}
    \label{eq: marginal density child born already}
        h_{ij}(v\,|\,u) = \frac{g_{ij}(v)}{\int_0^u g_{ij}(w)\,\mathrm{d}w}.
    \end{align}
\end{enumerate}
\end{lemma}
\begin{proof}
To prove part~(i), fix $m\in[d]$.
Then, conditional on $\{K_{mj}(u)=n\}$, the number of first-generation events, the sub-clusters $\bm{S}^\star_m(u-T_k)$ can be considered clusters
generated by an immigrant in component $m\in[d]$, which are i.i.d.\ due to the construction in part~(2) of Definition \ref{def: hawkes cluster}, modulo the time shift.

To prove (ii), we note that $K_{ij}(t)$ is distributed as a Poisson random variable with parameter $\int_0^t B_{ij}g_{ij}(s)\mathrm{d}s$, conditional on the realisation of $B_{ij}$;
see Definition \ref{def: hawkes cluster}.
Using Bayes rule, we compute
\begin{align*}
    H_{ij}(v\, | \, u)
    &= \pp( K_{ij}(v) = 1, K_{ij}(u) - K_{ij}(v) = 0 \ | \ K_{ij}(u) = 1 ) \\
    &= \frac{\exp\big( -\int_0^v B_{ij} g_{ij}(w)\mathrm{d}w\big) \int_0^v B_{ij} g_{ij}(w)\mathrm{d}w \exp\big( -\int_v^u B_{ij} g_{ij}(w)\mathrm{d}w\big)}{\exp\big( -\int_0^u B_{ij} g_{ij}(w)\mathrm{d}w\big) \int_0^u B_{ij} g_{ij}(w)\mathrm{d}w} \\
    &= \frac{\int_0^v g_{ij}(w)\mathrm{d}w}{\int_0^u g_{ij}(w)\,\mathrm{d}w},
\end{align*}
which yields the stated result.
\end{proof}

We proceed to our characterization of the transform $\bm{\cJ}_{\bm{S}^{\bm{Q}},\bm{S}^{\bm{\lambda}}}(u)$, for given $\bm{s},\bm{z}$,
in terms of a fixed point of the mapping $\phi$.
Here, by `fixed point' we mean that there exists an element in $\jj^d$ such that applying the mapping $\phi$ to it returns the same element.
(Uniqueness considerations will be dealt with later.)

\begin{theorem}
\label{thm: fixed point theorem}
The vector of time-dependent joint transforms ${\boldsymbol {\mathcal G}}(u)\equiv\bm{\cJ}_{\bm{S}^{\bm{Q}},\bm{S}^{\bm{\lambda}}}(u)$ defined in \eqref{eq: joint transform vector S^Q,S^lambda} satisfies the fixed-point equation 
\begin{align}
\label{eq: fixed point cluster process - phi}
    {\boldsymbol {\mathcal G}}(u)
    =\phi({\boldsymbol {\mathcal G}})(u).
\end{align}
\end{theorem}

\begin{proof}
The structure of the proof is as follows.
We start with the law of total expectation, then use the i.i.d.\ nature of (sub)-clusters,
write out the distribution of the first-generation events, and then collect terms.

We fix $j\in[d]$, and show that Eqn.\ (\ref{eq: fixed point cluster process - phi}) holds for the entry $\phi_j({\boldsymbol {\mathcal G}})$.
We throughout keep $\bm{s},\bm{z}$ fixed.
We start the computation of ${\mathcal G}_j(u)$ by applying the tower property and conditioning on the number of first-generation events.
For brevity, introduce the vector $\bm{K}_j(u) = (K_{1j},\dots,K_{dj}(u))^\top$, and note that $\pp(\bm{K}_j(u) = \bm{n}) = \prod_{m=1}^d \pp(K_{mj}(u)=n_m)$.
Using the distributional equalities (\ref{eq: distributional equality S_ij}),
in combination with the conditional independence between $S^{\bm{Q}}_{i\leftarrow j}(u)$, $S^{\bm{\lambda}}_{i\leftarrow j}(u)$ and $K_{mj}(u)$,
we have
\begin{align}\nonumber
    {\mathcal G}_j(u)&= \ee\Big[ \sum_{\bm{n}\in\nn_0^d} \ee\Big[\prod_{i=1}^d z_i^{S^{\bm{Q}}_{i\leftarrow j}(u)} e^{- s_i S^{\bm{\lambda}}_{i\leftarrow j}(u)} \, \Big| \, \bm{K}_j(u) = \bm{n}\Big] \,\pp(\bm{K}_j(u) = \bm{n}) \Big]\\
    &= c(u) \ee\Big[\sum_{\bm{n}\in\nn_0^d} \ee\Big[ \prod_{i=1}^d z_i^{\sum\limits_{m=1}^d \sum\limits_{k=1}^{n_{m}}S^{\bm{Q}}_{i\leftarrow m}(u-T_{k})} e^{-s_i \sum\limits_{m=1}^d \sum\limits_{k=1}^{n_{m}} S^{\bm{\lambda}}_{i\leftarrow m}(u-T_{k})} \Big] \,\pp(\bm{K}_j(u) = \bm{n}),\label{eq:gju}
\end{align}
where, for brevity, we introduced the constant
\[c(u):=  \ee\Big[z_j^{\bm{1}_{\{J_j>u\}}}\Big]\prod_{i=1}^d \ee\Big[e^{-s_iB_{ij}g_{ij}(u)}\Big].\]
We now use the i.i.d.\ nature of the sub-clusters, as described by Lemma \ref{lem: cluster independence children},
to write the inner expectation in \eqref{eq:gju} as a product over the source components of the first-generation events.
To that end, let $T^{(mj)}$ be a random variable with probability density function $h_{mj}(\cdot)$ as given in (ii) of Lemma \ref{lem: cluster independence children},
such that $T^{(mj)}$ is distributed as $T_k$ if it was generated by $K_{mj}(\cdot)$.
With these observations and the definition of ${\boldsymbol {\mathcal G}}(\cdot)$, we can write
\begin{align*}
    \ee\Big[ \prod_{i=1}^d &z_i^{\sum\limits_{m=1}^d \sum\limits_{k=1}^{n_{m}}S^{\bm{Q}}_{i\leftarrow m}(u-T_{k})} e^{-s_i \sum\limits_{m=1}^d \sum\limits_{k=1}^{n_{m}} S^{\bm{\lambda}}_{i\leftarrow m}(u-T_{k})} \Big] \\
    &=\prod_{m=1}^d\ee\Big[ \prod_{i=1}^d z_i^{S^{\bm{Q}}_{i\leftarrow m}(u-T^{(mj)})} e^{-s_iS^{\bm{\lambda}}_{i\leftarrow m}(u-T^{(mj)})} \Big]^{n_{m}}
    = \prod_{m=1}^d \Big({\mathcal G}_m(u-T^{(mj)})\Big)^{n_{m}}.
\end{align*}
Using that the $K_{mj}(\cdot)$ are Poisson processes with intensity $B_{mj}g_{mj}(\cdot)$, and writing out the density $h_{mj}(\cdot)$,
we thus find
\begin{align*}
    {\mathcal G}_j&(u)
    =c(u)\, \ee\Big[\sum_{\bm{n}\in\nn_0^d} \prod_{m=1}^d \Big(\int_0^u h_{mj}(v\,|\,u)\,{\mathcal G}_m(u-v)\mathrm{d}v\Big)^{n_{m}} \\
    &\quad\hspace{1cm}\times\frac{\big(B_{mj}\int_0^u g_{mj}(v)\,\mathrm{d}v\big)^{n_{mj}}}{n_{m}!}
    \exp\Big(-B_{mj}\int_0^u g_{mj}(v)\,\mathrm{d}v\Big)\Big]\\
    &=c(u)\, \ee\Big[\sum_{\bm{n}\in\nn_0^d} \prod_{m=1}^d \frac{1}{n_{m}!}\Big(B_{mj}\int_0^u g_{mj}(v)\,{\mathcal G}_m(u-v)\,\mathrm{d}v\Big) ^{n_{m}}\, \exp\Big(-B_{mj}\int_0^u g_{mj}(v)\,\mathrm{d}v\Big)\Big] \\
    &=c(u)\prod_{m=1}^d\ee\Big[  \exp\Big( B_{mj} \int_0^u g_{mj}(v)\big({\mathcal G}_m(u-v) - 1)\mathrm{d}v\Big)\Big],
\end{align*}
where the last equality holds due to independence between the random variables $B_{mj}$.
Note that this last expression equals ${\mathcal G}_j(u)=\phi_j(\bm{\cJ}_{\bm{S}^{\bm{Q}},\bm{S}^{\bm{\lambda}}})(u,\bm{s},\bm{z})$ as was introduced in (\ref{eq: phi_j definition mapping}),
which finishes the proof.
\end{proof}

The functional equation described in Eqn.\ \eqref{eq: fixed point cluster process - phi} can be exploited to numerically approximate the joint transforms of the cluster processes.
The convergence result of Theorem \ref{thm: convergence of phi iteration} below entails that iterating the map $\phi$ leads to the desired fixed point,
thus having found ${\boldsymbol {\mathcal G}}(u)$ uniquely.
Once ${\boldsymbol {\mathcal G}}(u)$ has been obtained,
numerical inversion can be applied to obtain the corresponding joint densities and distribution functions;
likewise, arbitrary joint spatial-temporal moments can be evaluated by differentiation.

We now proceed to establish the convergence result.
Consider a joint transform $\bm{\cJ}^{(0)}(\cdot)\in\jj^d$.
Define the sequence $(\bm{\cJ}^{(n)}(\cdot))_{n\in\nn_0}$ by $\bm{\cJ}^{(n)}(\cdot) := \phi(\bm{\cJ}^{(n-1)}(\cdot))$ for $n\in{\mathbb N}$,
where $\phi$ is the mapping in Eqn.\ (\ref{eq: phi_j definition mapping}).
Note that $\bm{\cJ}^{(n)}(\cdot) \in \jj^d$ for all $n\in\nn_0$ by Lemma \ref{lem: phi well defined} and induction.

\begin{theorem}
\label{thm: convergence of phi iteration}
For any $\bm{\cJ}^{(0)}(\cdot)\in\jj^d$, the sequence $(\bm{\cJ}^{(n)})_{n\in\nn_0}(u)$ converges pointwise to the fixed point ${\boldsymbol {\mathcal G}}(u)=\bm{\cJ}_{\bm{S}^{\bm{Q}},\bm{S}^{\bm{\lambda}}}(u)$, i.e., as $n\to\infty$, for any $u\leqslant t$,
\begin{align}
    \bm{\cJ}^{(n)}(u)\equiv \bm{\cJ}^{(n)}(u,\bm{s},\bm{z}) \to \bm{\cJ}_{\bm{S}^{\bm{Q}},\bm{S}^{\bm{\lambda}}}(u,\bm{s},\bm{z})
    \equiv \bm{\cJ}_{\bm{S}^{\bm{Q}},\bm{S}^{\bm{\lambda}}}(u).
\end{align}
\end{theorem}

\begin{proof}
Consider, for $\bm{\cJ}_{\rm A}^{(0)}(\cdot),{\bm{\cJ}}_{\rm B}^{(0)}(\cdot) \in\jj^d$, the sequences $\bm{\cJ}_{\rm A}^{(n)}(\cdot),{\bm{\cJ}}_{\rm B}^{(n)}(\cdot) \in\jj^d$ by
\[\bm{\cJ}_{\rm A}^{(n)}(\cdot) := \phi(\bm{\cJ}_{\rm A}^{(n-1)}(\cdot)),\hspace{1.2cm}
\bm{\cJ}_{\rm B}^{(n)}(\cdot) := \phi(\bm{\cJ}_{\rm B}^{(n-1)}(\cdot)),\]
where $n\in{\mathbb N}$.
We show that the sequences have a unique limit by first proving that there exists a constant $M_j>0$ such that, uniformly in $n\in{\mathbb N}_0$ and $u\leqslant t$,
\begin{align}
\label{eq: inductive step}
    | \big(\bm{\cJ}_{\rm A}^{(n)}\big)_j(u) - \big({\bm{\cJ}}_{\rm B}^{(n)}\big)_j(u)| \leqslant \frac{1}{n!}(M_ju)^n,
\end{align}
where $\big({\bm{\cJ}}_{i}^{(n)}\big)_j(u)$ it the $j$-th entry of $\bm{\cJ}_{i}^{(n)}(u)$, for $i\in\{{\rm A},{\rm B}\}$.
We prove \eqref{eq: inductive step} inductively.
For $n=1$, using the bounds appearing in the proof of Lemma \ref{lem: phi continuous and bounds}, we have
\begin{align*}
    | \big(\bm{\cJ}_{\rm A}^{(1)}\big)_j(u) - \big({\bm{\cJ}}_{\rm B}^{(1)}\big)_j(u) |
    &\leqslant d \max_{i\in[d]} \ee[B_{ij}]\,\lVert g_{ij} \rVert_{L^1(\rr_+)}u = M_j u,
\end{align*}
where we choose $M_j := d \max_{i\in[d]} \ee[B_{ij}]\,\lVert g_{ij} \rVert_{L^1(\rr_+)}$, which is finite by assumption.
For the induction step, assume that (\ref{eq: inductive step}) holds for some $n\in\nn$.
Again by the bounds appearing in the proof of Lemma \ref{lem: phi continuous and bounds}, we have
\begin{align*}
    |\big(\bm{\cJ}_{\rm A}^{(n+1)}\big)_j(u) &- \big({\bm{\cJ}}_{\rm B}^{(n+1)}\big)_j(u)| \\
    &\leqslant \sum_{m=1}^d\ee\Big[B_{mj}\int_0^u g_{mj}(v) \,\mathrm{d}v\int_0^u|\big(\bm{\cJ}_{\rm A}^{(n+1)}\big)_m(v) - \big({\bm{\cJ}}_{\rm B}^{(n+1)}\big)_m (v)| \,\mathrm{d}v \Big]\\
    &\leqslant M_j\frac{1}{n!}\int_0^u (M_jv)^n\,\mathrm{d}v = \frac{1}{(n+1)!}(M_j u)^{n+1}.
\end{align*}
Since this holds for all $j\in[d]$, it is clear that \eqref{eq: inductive step} implies that
\begin{align*}
    \lVert \bm{\cJ}_{\rm A}^{(n)}(u) - {\bm{\cJ}}^{(n)}_{\rm B}(u)\rVert_{\rr^d}^2
    \leqslant \sum_{j=1}^d \frac{1}{n!}(M_ju)^{2n}
    \leqslant \frac{d}{n!} (Mu)^{2n},
\end{align*}
where $M := \max_{j\in[d]} M_j$. Letting $n\to\infty$, we note that both sequences have the same limit
\begin{align}
\label{eq: limit joint transforms exists}
    \lim_{n\to\infty} \bm{\cJ}^{(n)}_{\rm A}(u) = \lim_{n\to\infty} {\bm{\cJ}}^{(n)}_{\rm B}(u) = \bm{\cJ}(u).
\end{align}

We now show that $\bm{\cJ}(\cdot) \in \jj^d$, which we argue to hold by applying L\'evy's continuity theorem.
Note that $(\bm{\cJ}^{(n)})_j(u)$ can be considered as a characteristic function;
indeed, as a consequence of $\bm{\cJ}^{(n)}(\cdot) \in \jj^d$, we can rewrite,
for certain $Y_m(u)\equiv Y_m^{(j,n)}(u)$, $X_m(u)\equiv X_m^{(j,n)}(u)$, and
$\bm{Z}^{(j,n)}(u):=\bm{Y}^{(j,n)}(u)-\bm{X}^{(j,n)}(u)$,
\begin{align*}
    \big(\bm{\cJ}^{(n)}\big)_j(u)
    &= \ee\Big[ \prod_{m=1}^d z_m^{Y_m(u)} e^{-s_m X_m(u)} \Big] \\
    &= \ee\Big[ \prod_{m=1}^d e^{{\rm i}\, r_m Y_m(u)-{\rm i}\,r_m X_m(u) }\Big] =: \ee\Big[ \exp\big( {\rm i}\, \bm{r}^\top \bm{Z}^{(j,n)}(u)\big)\Big],
\end{align*}
where we set, for $m\in[d]$, $s_m = {\rm i}\,r_m$ and $z_m = e^{{\rm i}\,r_m}$
with ${\rm i}$ the imaginary unit and $r_m \in \rr$.
As a result,
$(\bm{\cJ}^{(n)})_j(u)$ is the characteristic function of the random vector $\bm{Z}^{(j,n)}(u)$.
Since we established above that the limit of $\bm{\cJ}^{(n)}(u)$, as $n\to\infty$, exists,
L\'evy's continuity theorem implies that there is a random variable $\bm{Z}^{(j)}$
such that $\bm{Z}^{(j,n)}(u)$ converges weakly to $\bm{Z}^{(j)}$ as $n\to\infty$ with the  transform of $\bm{Z}^{(j)}$ being $(\bm{\cJ})_j(u)$.
This implies $(\bm{\cJ})_j(\cdot)\in\jj$, and hence also $\bm{\cJ}(\cdot)\in\jj^d$.

Finally, by Lemma \ref{lem: phi continuous and bounds} we know that $\phi$ is continuous, so
\begin{align*}
    \bm{\cJ}(u) = \lim_{n\to\infty} \bm{\cJ}^{(n+1)}(u) &= \lim_{n\to\infty}\phi\big(\bm{\cJ}^{(n)}\big)(u) = \phi\big(\lim_{n\to\infty}\bm{\cJ}^{(n)}\big)(u) = \phi(\bm{\cJ})(u),
\end{align*}
implying that the limit point is a fixed point of $\phi$.
By Theorem \ref{thm: fixed point theorem} we know that $\bm{\cJ}_{\bm{S}^{\bm{Q}},\bm{S}^{\bm{\lambda}}}(\cdot)$ is a fixed point of $\phi$ and combined with (\ref{eq: limit joint transforms exists}),
where we derived that every sequence has the same (unique) limit,
we have that $\bm{\cJ}^{(n)}(u) \to \bm{\cJ}_{\bm{S}^{\bm{Q}},\bm{S}^{\bm{\lambda}}}(u)$ as $n\to\infty$, irrespective of $\bm{\cJ}^{(0)}(\cdot)\in\jj^d$.
\end{proof}


\section{Tail Probabilities}\label{sec:TailProbs}
In the previous sections, we have provided an exact analysis of the probabilistic behavior of (primarily) the random object $(\bm{Q}(t),\bm{\lambda}(t))$ for $t>0$.
We have characterized in particular the associated joint transform and established a corresponding fixed-point representation,
allowing for general decay functions $g_{ij}(\cdot)$,
general distributions of the jump sizes $B_{ij}$, and
general distributions of the sojourn times $J_i$.
In the present section, we provide an asymptotic analysis pertaining directly to the probability distribution functions of $\bm{Q}(t)$ and $\bm{\lambda}(t)$,
in the setting where the jump sizes' tail distributions are essentially of a power-law nature.

\subsection{Power-law tails and the Hawkes graph}
We start our exposition by introducing the concept of asymptotically power-law tails.
\begin{definition}[\textsc{Asymptotically Power-Law Tail (APT)}]\label{def:APT}
We say that a scalar-valued non-negative random variable $X$ has an asymptotically power-law tail
if there exist positive constants $C$ and $\gamma$ such that
\[{\mathbb P}(X>x) \,x^\gamma \to C,\]
as $x\to\infty$.
We write: $X\in {\rm APT}(C,\gamma)$ and refer to $\gamma$ as the \textit{tail index}.
\end{definition}

In the sequel, we assume that, for each $i,j\in[d]$, either $B_{ij}\in {\rm APT}(C_{ij},\gamma_{ij})$ for constants $C_{ij}>0$ and $\gamma_{ij}>1$, or $B_{ij}\equiv 0$.
Later in this section we discuss a few generalizations.

In the literature, a substantial amount of attention has been devoted to probabilistic systems in which
some of the underlying random variables have a distribution function with a power-law tail.
One often works with a class of distributions that is closely related to, but slightly wider than, APT,
namely the class of \textit{regularly varying} distributions:
then, ${\mathbb P}(X>x)$ behaves as $\ell(x)\,x^{-\gamma}$ when $x\to\infty$,
for some slowly-varying function $\ell(\cdot)$ (i.e., for all $a>0$ we have that $\ell(ax)/\ell(x) \to 1$).
A detailed exposition of regular variation in the context of insurance and finance can, for instance, be found in the monograph \cite{EKM97},
and for examples of its use in queueing theory we refer to \cite{FKZ07, ZBM04};
a general treatment is in \cite{BGT87}.
A powerful concept in this branch of the literature is the so-called `principle of a single big jump':
in many systems, rare events happen with `overwhelming' probability due to a single extreme outcome of a random quantity that features in the model.
In a simple single-dimensional counterpart of our model,
it is shown in \cite{KSBM18} that in the spirit of this principle $Q(t)$ essentially inherits the tail behavior of the jump size $B$.

An important result that we exploit in the general context of this section,
is a closure property related to the sum of independent random variables in APT.
If $X_i$ is in APT with tail index $\gamma_i$, for $i=1,2$, and $X_1$ and $X_2$ are independent,
then $X_1+X_2$ is also in APT with tail index equal to $\min\{\gamma_1,\gamma_2\}$, i.e., the heaviest tail dominates; cf.\ \cite{EKM97}.
Based on this property, one could na\"{\i}vely guess that in the present multivariate setting, the tail of $Q_i(t)$ will resemble the tail of the heaviest among $B_{i1},\ldots, B_{id}$.
This is however not necessarily the case: due to potential cross-excitation, heavy tails originating in another component may \textit{indirectly} propagate to component $i$.
This concept of propagation can be conveniently reasoned about relying on so-called Hawkes graphs, defined in the present setting as follows (see also, e.g., \cite{BMM15,K17} for related, different definitions).

\begin{definition}[\textsc{Hawkes graph}]
Let $V=\{1,\ldots,d\}$ be a set of vertices.
Let an edge $e_{ij}$ from $j$ to $i$ exist if $B_{ij}\in {\rm APT}(C_{ij},\gamma_{ij})$ for $C_{ij}>0$, $\gamma_{ij}>1$
(i.e., $B_{ij}$ is not identical to $0$),
and call the resulting set of directed edges $E$.
Then, the directed graph $(V,E)$ is called the Hawkes graph.
\end{definition}

Note that the vertices $V$ of the Hawkes graph are associated to the components of the Hawkes process $(\bm{Q}(\cdot), \bm{\lambda}(\cdot), \bm{N}(\cdot))$,
i.e., $i\in V$ corresponds to $Q_i(t)$ (or $\lambda_i(t), N_i(t))$.
The states in the Hawkes graph $(V,E)$ can be classified similarly to how this is conducted for Markov chains.
To this end, we set $P_{i\leftarrow j}=1$ if a path from $j$ to $i$ in $(V,E)$ exists, and $P_{i\leftarrow j}=0$ otherwise.
(The order of the indices in $P_{i\leftarrow j}$ might look unnatural at first sight,
but it should be borne in mind that $B_{ij}\in {\rm APT}(C_{ij},\gamma_{ij})$ implies that there is an edge from $j$ to $i$ in the Hawkes graph,
with corresponding jump sizes generated according to the random variable $B_{ij}$.)
We say that vertices $i$ and $j$ belong to the same class if $P_{i \leftarrow j}=P_{j\leftarrow i}=1$.
We shall call a class \textit{recurrent} if there is no path to vertices outside the class, otherwise it is \textit{transient}.

\begin{example}[\textsc{Bivariate, triangular}]
\label{example: bivariate hawkes graph}
Let $d=2$ and assume that $B_{12}$ is identical to $0$.
The corresponding Hawkes graph can be visualized as in the figure below.
\begin{center}
\begin{tikzpicture}[->,thick,node distance = 3cm, roundnode/.style={circle, draw= black, thick, minimum size = 7mm}]
\node[roundnode] (Q1) at (0,0) {$Q_1$};
\node[roundnode][right of =Q1] (Q2) {$Q_2$};
\draw[black] (Q1) to [out=30,in=150] node[above]{$e_{21}$} (Q2);
\draw[black] (Q1) to [out=150,in = 210,looseness = 7] node[left]{$e_{11}$}  (Q1);
\draw[black] (Q2) to [out=30,in = 330,looseness = 7] node[right]{$e_{22}$} (Q2);
\end{tikzpicture}
\end{center}
In this example, $P_{1\leftarrow 1} = 1$, $P_{2\leftarrow 1} = 1$, and $P_{2\leftarrow 2} =1$, but $P_{1\leftarrow 2} = 0$.
We conclude that there is one transient class $\{1\}$ and one recurrent class $\{2\}$.
\end{example}

\begin{example}[\textsc{Trivariate}]
\label{example: trivariate hawkes graph}
Let $d=3$ and assume that $B_{12}$, $B_{13}$, $B_{22}$ and $B_{31}$ are identical to $0$.
Then, the Hawkes graph takes the following form:
\begin{center}
\begin{tikzpicture}[->,thick,node distance = 3cm, roundnode/.style={circle, draw= black, thick, minimum size = 7mm}]
\node[roundnode] (Q1) at (0,0) {$Q_1$};
\node[roundnode][right of =Q1] (Q2) {$Q_2$};
\node[roundnode][right of =Q2] (Q3) {$Q_3$};
\draw[black] (Q1) to [out=30,in=150] node[above]{$e_{21}$} (Q2);
\draw[black] (Q1) to [out=150,in = 210,looseness = 7] node[left]{$e_{11}$}  (Q1);
\draw[black] (Q2) to [out=30,in=150] node[above]{$e_{32}$} (Q3);
\draw[black] (Q3) to [out=30,in = 330,looseness = 7] node[right]{$e_{33}$} (Q3);
\draw[black] (Q3) to [out=210,in=330] node[below]{$e_{23}$} (Q2);
\end{tikzpicture}
\end{center}
It is directly seen that there is one transient class $\{1\}$ and one recurrent class $\{2,3\}$,
since there is no path to $1$ from the other vertices.
\end{example}

The following lemma relates the cluster processes to the Hawkes graph path indicators.
The proof follows using standard techniques and is for completeness given in Section \ref{section: proofs} in the Appendix.

\begin{lemma}
\label{lem: equivalence R_ij > 0 and path existence}
Consider $u>0$. The following statements are equivalent:\\
{\em (i)}~${\mathbb E}[S^{\bm{Q}}_{i\leftarrow j}(u)]>0$; \\
{\em (ii)}~${\mathbb E}[S^{\bm{\lambda}}_{i\leftarrow j}(u)]>0$; and\\
{\em (iii)}~$P_{i\leftarrow j} =1$.
\end{lemma}

\begin{proof}
See Section \ref{section: proofs} in the Appendix.
\end{proof}

\subsection{Tails of the marginal distributions}\label{subsec:marginaltail}
In this subsection, we establish the asymptotic behavior of ${\mathbb P}(Q_i(t)>x)$ and ${\mathbb P}(\lambda_i(t)>x)$
as $x\to\infty$, for any $i\in[d]$.
We start by introducing a few objects that play a pivotal role in our analysis:
\begin{align*}
    \delta_{ij}&:= \min_{m\in[d]}\{\gamma_{mj}: P_{i\leftarrow m}=1\},\:\:\:
    \bar\gamma_i := \min_{j\in[d]}\delta_{ij},\:\:\:
     I_i := \argmin_{j\in[d]} \delta_{ij},\:\:\:
     I_{ij} := \argmin_{m\in[d]}\{\gamma_{mj}: j\in I_i\}.
\end{align*}
Given the existence of a path from $m$ to $i$, $\delta_{ij}$ determines the smallest $\gamma_{mj}$ associated to $B_{mj}$
over all such $m$ for a given $j$.
Here, it is noted that we do not assume that $\min_{j\in[d]}\delta_{ij}$ is attained at a unique argument,
i.e., $I_i$ is a set that potentially consists of more than one element; the same applies to $I_{ij}$.

In addition, we introduce two functionals, in the same spirit as Eqn.~\eqref{eq: def linear functional A_ij branching structure first gen}.
First, with $\bm{x}(\cdot)=(x_1(\cdot),\ldots,x_d(\cdot))$ a (row-)vector-valued function and $P\geqslant 0$, for $j\in[d]$,
\begin{align}
    {\mathcal B}_{j}\big\{P, \bm{x}(\cdot)\big\}(u) = P
    +\sum_{m=1}^d {\mathbb E}[B_{mj}] \int_0^u g_{mj}(v)\,x_{m}(u-v)\,{\rm d}v.
    \label{eq:B_j}
\end{align}
Second, with $\bm{y}(\cdot)=(y_1(\cdot),\ldots,y_d(\cdot))$ denoting another (row-)vector-valued function and $\delta\in(1,2)$, for $i,j\in[d]$,
\begin{align}\nonumber
    {\mathcal B}^\delta_{ij}\big\{P, \bm{x}(\cdot)\,|\,\bm{y}(\cdot)\}(u) &=
    \omega_i P+ \sum_{m\in I_i}{\mathbb E}[B_{mj}]\int_0^u g_{mj}(v)\,x_{m}(u-v)\,{\rm d}v\,\\
    &\:\:\:\hspace{2cm}+\:\omega_i \sum_{m\in I_{ij}} C_{mj} \left(\int_0^u g_{mj}(v)\,y_{m}(u-v)\,{\rm d}v\right)^{ \delta},
    \label{eq:B_ij^delta}
\end{align}
where $\omega_i=\Gamma(1-\delta)$ with $\Gamma(\cdot)$ the Gamma function.

In the sequel, we will intensively work with functions $R^\star_{ij}(\cdot)$ and $R^{\star,\delta}_{ij}(\cdot)$, $\star\in\{\bm{Q},\bm{\lambda}\}$,
defined using \eqref{eq:B_j}--\eqref{eq:B_ij^delta},
and their row-vector-valued counterparts $\bm{R}^{\star}_{(i)}(\cdot)$ and $\bm{R}^{\star,\delta}_{(i)}(\cdot)$,
defined analogously to $\bm{S}^\star_{(i)}(\cdot)$.
The functions $R^{\bm{Q}}_{ij}(\cdot)$ and $R^{\bm{\lambda}}_{ij}(\cdot)$ for $i,j\in[d]$ satisfy the system of coupled functional equations
\begin{align}
\label{eq: R_ij linear term equation}
\begin{split}
    R^{\bm{Q}}_{ij}(u) &= {\mathcal B}_{j}\big\{\bm{1}_{\{i=j\}}\,{\mathbb P}(J_i>u), \bm{R}_{(i)}^{\bm{Q}}(\cdot)\big\}(u),\:\:\:\:\:\:
    R^{\bm{\lambda}}_{ij}(u) = {\mathcal B}_{j}\big\{{\mathbb E}[B_{ij}]\,g_{ij}(u), \bm{R}_{(i)}^{\bm{\lambda}}(\cdot) \big\}(u),
\end{split}
\end{align}
whereas the functions $R_{ij}^{\bm{Q},\delta}(\cdot)$ and $R_{ij}^{\bm{\lambda},\delta}(\cdot)$, with $j\in I_i$ and $\delta\in(1,2)$,
satisfy the system of coupled functional equations
\begin{align}
\label{eq: R_ij^gamma fractional term equation}
\begin{split}
    R^{\bm{Q},\delta}_{ij}(u) &= {\mathcal B}_{ij}^{\delta}\big\{0,\bm{R}_{(i)}^{\bm{Q},\delta}(\cdot)\,|\,
    \bm{R}_{(i)}^{\bm{Q}}(\cdot)
    \big\}(u), \:\:\:\:\:\:
    R^{\bm{\lambda},\delta}_{ij}(u) =  {\mathcal B}_{ij}^{\delta}\big\{\omega_iC_{ij},\bm{R}_{(i)}^{\bm{\lambda},\delta}(\cdot)\,|\,
    \bm{R}_{(i)}^{\bm{\lambda}}(\cdot)\big\}(u).
\end{split}
\end{align}

The following theorem reveals how APT-behavior is inherited by $Q_i(t)$ and $\lambda_i(t)$, $i\in[d]$.
Henceforth, for positive functions $f(\cdot)$ and $g(\cdot)$, we write $f(x) \sim g(x)$ as $x\to x_0$ to mean $\lim_{x\to x_0} f(x)/g(x) = 1$.

\begin{theorem}
\label{thm: Q_i and lambda_i heavy tailed}
Fix $i\in[d]$ and $t \in \rr_+$.
Assume that $\bar\gamma_i\in(1,2)$.
Then, $Q_i(t) \in {\rm APT}(\bar C^{\bm{Q}}_i,\bar\gamma_i)$ and
$\lambda_i(t) \in {\rm APT}(\bar C^{\bm{\lambda}}_i,\bar\gamma_i)$ for some $\bar C^{\bm{Q}}_i, \bar C^{\bm{\lambda}}_i>0$. More precisely,
\begin{align}
\label{eq: z transform Q_i and lambda_i heavy tailed thm statement}
\begin{split}
    &{\mathbb E}\Big[z^{ Q_i(t)}\Big] - 1 + (1-z)\, {\mathbb E}\big[Q_i(t)\big] \sim -(1-z)^{\bar\gamma_i} \sum_{j\in I_i}\lambda_{j,\infty} \int_0^t R^{\bm{Q},\bar\gamma_i}_{ij}(u)\,\mathrm{d}u, \notag \\
    &{\mathbb E}\Big[e^{ -s\lambda_i(t)}\Big] - 1 +s \,{\mathbb E}\big[\lambda_i(t)\big] \sim -s^{\bar\gamma_i}
    \sum_{j\in I_i}\lambda_{j,\infty}\int_0^t R^{\bm{\lambda},\bar\gamma_i}_{ij}(u)\,\mathrm{d}u,
\end{split}
\end{align}
as $z\uparrow 1$ and $s\downarrow 0$, respectively, and the first moments equal \[{\mathbb E}\big[Q_i(t)\big] = \sum_{j=1}^d \lambda_{j,\infty} \int_0^t R^{\bm{Q}}_{ij}(u)\,\mathrm{d} u,\:\:\:\:{\mathbb E}\big[\lambda_i(t)\big] = \sum_{j=1}^d\lambda_{j,\infty} \int_0^t R^{\bm{\lambda}}_{ij}(u)\,\mathrm{d} u.\]
\end{theorem}

Before we give the proof of Theorem~\ref{thm: Q_i and lambda_i heavy tailed},
we discuss the systems of Eqns.\ \eqref{eq: R_ij linear term equation}--\eqref{eq: R_ij^gamma fractional term equation}.
First, observe that they can be considered as vector-valued \textit{renewal} equations, owing to the structure of \eqref{eq:B_j}--\eqref{eq:B_ij^delta}.
We point out how to solve these for \eqref{eq: R_ij linear term equation};
\eqref{eq: R_ij^gamma fractional term equation} can be dealt with analogously, hence we only provide its final result.
In the sequel, we denote the Laplace-Stieltjes transform of $f(\cdot)$ by
\[\cL\{f(\cdot)\}(r) = \int_0^\infty e^{-ru}\,f(u)\,\mathrm{d}u,\]
for $r\geqslant 0$.
Taking the Laplace-Stieltjes transform of \eqref{eq: R_ij linear term equation}, recognizing the convolution structure, we readily obtain, with $\bar F_i(\cdot)={\mathbb P}(J_i>\cdot\,)$,
\begin{align}
\label{eq: Laplace-Stieltjes transform equations R_ij^Q}
    \cL\big\{R_{ij}^{\bm{Q}}(\cdot)\big\}(r) &= \bm{1}_{\{i=j\}}\cL\big\{\bar F_i(\cdot)\big\}(r) + \sum_{m=1}^d {\mathbb E}[B_{mj}]\cdot \cL\big\{g_{mj}(\cdot)\big\}(r) \cdot\cL\big\{R_{im}^{\bm{Q}}(\cdot)\big\}(r).
\end{align}
For a given argument $r\geqslant 0$ and component $i\in[d]$,
Eqn.\ \eqref{eq: Laplace-Stieltjes transform equations R_ij^Q} is a linear system from which the unknowns $z_{ij}(r)=\cL\{R_{ij}^{\bm{Q}}(\cdot)\}(r)$ can be solved.
This is done by first solving for any recurrent class the corresponding linear subsystem,
and then iteratively for any transient class leading to these recurrent classes.
Having computed $z_{ij}(\cdot)$, standard Laplace inversion can be invoked to identify the functions $R_{ij}^{\bm{Q}}(\cdot)$; see also Section \ref{sec:Numerics}.

A similar argumentation applies when taking the Laplace-Stieltjes transform of
\eqref{eq: R_ij^gamma fractional term equation}, leading to,
\begin{align}\notag
    \cL\big\{R_{ij}^{\bm{Q},\delta}(\cdot)\big\}(r) &= \sum_{m \in I_i} {\mathbb E}[B_{mj}]\cdot \cL\big\{g_{mj}(\cdot)\big\}(r)\cdot \cL\big\{R_{im}^{\bm{Q},\delta}(\cdot)\big\}(r) \\
    &\hspace{15mm}+ \omega_i \sum_{m \in I_{ij}} C_{mj}\, \cL\Big\{\big(g_{mj} *R_{im}^{\bm{Q}}\big)^{\delta}(\cdot)\Big\}(r), \label{eq: Laplace-Stieltjes transform equations R_ij^Q,gamma}
\end{align}
where, as usual, $*$ denotes the convolution operator, i.e.,
\begin{align*}
    \big(g_{mj} *R_{im}^{\bm{Q}}\big)^{\delta}(t) = \left(\int_0^t g_{mj}(v)R_{im}^{\bm{Q}}(t-v)\mathrm{d}v\right)^{\delta}.
\end{align*}
This is again a linear system, from which the $\bar z_{ij}(r)=\cL\{R_{ij}^{\bm{Q},\delta}(\cdot)\}(r)$ can be solved for any given $r\geqslant 0$.
Recall that at this stage the functions $R_{im}^{\bm{Q}}(\cdot)$ are available, thus allowing us to evaluate the last term in \eqref{eq: Laplace-Stieltjes transform equations R_ij^Q,gamma}.
We can obtain similar equations for the transforms
pertaining to $\lambda_i(t)$; as these are fully analogous, we leave them out.

\begin{proof}[Proof of Theorem \ref{thm: Q_i and lambda_i heavy tailed}]
We fix $i\in[d]$ and prove the result for $Q_i(t)$; the arguments for $\lambda_i(t)$ are similar, \textit{mutatis mutandis}.
To that end, we set $s_m = 0$ for all $m\in[d]$ and $z_k = 1$ for $k\neq i$ and $z_i \equiv z \in [-1,1]$ in the joint transform of $({\bs Q}(t),{\bs \lambda}(t))$,
$\cJ_{\bm{Q},\bm{\lambda}}(t)$.
We thus obtain the $z$-transform of $Q_i(t)$ in terms of the cluster process entries $S^{\bm{Q}}_{i\leftarrow j}(u)$, 
\begin{align}
\label{eq: def z transform Q_i and fixed point}
\begin{split}
    \ee\Big[z^{Q_i(t)}\Big] &= \prod_{j=1}^d \exp\Big( \lambda_{j,\infty} \int_0^t\big( \ee\big[z^{S^{\bm{Q}}_{i\leftarrow j}(u)}\big] - 1\big)\mathrm{d}u\Big),\\
    \ee\big[z^{S^{\bm{Q}}_{i\leftarrow j}(u)}\big] &= \ee\Big[z^{\bm{1}_{\{i=j\}}\bm{1}_{\{J_i>u\}}}\Big] \prod_{m=1}^d \cT\{B_{mj}\}\Big(\int_0^u g_{mj}(v)\big(1-\ee\big[z^{S^{\bm{Q}}_{i\leftarrow m}(u-v)}\big] \big)\mathrm{d}v\Big),
\end{split}
\end{align}
where the latter equation holds by our fixed-point theorem, i.e., Theorem~\ref{thm: fixed point theorem},
and $\cT\{B_{mj}\}(s) = \ee[e^{-sB_{mj}}]$ denotes the Laplace-Stieltjes transform of $B_{mj}$.
The idea is to analyze expansions of the $z$-transform appearing in (\ref{eq: def z transform Q_i and fixed point}) by using Tauberian theorems,
so as to establish the tail behavior of $Q_i(t)$.

Under the assumption that $B_{mj} \in {\rm APT}(C_{mj},\gamma_{mj})$, the $B_{mj}$ that have index $\gamma_{mj} \in (1,2)$
satisfy
\begin{align}
\label{eq: bingham 8.1.6 for B_mj}
    \cT\{B_{mj}\}(r) \sim 1 - r\ee[B_{mj}] - C_{mj} \Gamma(1-\gamma_{mj})r^{\gamma_{mj}},
\end{align}
as $r\downarrow 0$, where $\Gamma(\cdot)$ is the Gamma function, by virtue of a Tauberian theorem; see e.g., \cite[Theorem 8.1.6]{BGT87}.
Observe that \[\int_0^u g_{mj}(v)\big(1- \ee\big[z^{S^{\bm{Q}}_{i\leftarrow m}(u-v)}\big] \big)\mathrm{d}v \downarrow 0\] as $z \uparrow 1$.
Using this in (\ref{eq: bingham 8.1.6 for B_mj}), we obtain
\begin{align}
\label{eq: Tauber B_mj in transform terms}
    &\cT\{B_{mj}\}\Big(\int_0^u g_{mj}(v)\big(1-\ee\big[z^{S^{\bm{Q}}_{i\leftarrow m}(u-v)}\big]\big)\mathrm{d}v\Big)\notag \\
    &\sim 1 - \ee[B_{mj}] \int_0^u g_{mj}(v)\big(1-\ee\big[z^{S^{\bm{Q}}_{i\leftarrow m}(u-v)}\big] \big)\mathrm{d}v \\
    &\quad - C_{mj}\Gamma(1-\gamma_{mj}) \Big(\int_0^u g_{mj}(v)\big(1-\ee\big[z^{S^{\bm{Q}}_{i\leftarrow m}(u-v)}\big] \big)\mathrm{d}v\Big)^{\gamma_{mj}}.\notag
\end{align}
We substitute (\ref{eq: Tauber B_mj in transform terms}) in the expression for $\ee\big[z^{S^{\bm{Q}}_{i\leftarrow j}(u)}\big]$ in \eqref{eq: def z transform Q_i and fixed point}
in combination with the calculation $\ee[z^{\bm{1}_{\{i=j\}}\bm{1}_{\{J_i>u\}}}] = \pp(J_i \leqslant u) + z^{\bm{1}_{\{i=j\}}}\pp(J_i>u) = 1 - (1-z^{\bm{1}_{\{i=j\}}})\pp(J_j>u)$. This allows us to write, up to $O((1-z)^2)$ terms,
\begin{align}
\label{eq: 1-eta_c with B_mj}
    1- \ee\big[z^{S^{\bm{Q}}_{i\leftarrow j}(u)}\big]
    &\sim  1- \big( 1 - (1-z^{\bm{1}_{\{i=j\}}})\pp(J_i>u)\big) \notag\\
    &\quad \times\prod_{m=1}^d \Big\{ 1 - \ee[B_{mj}] \int_0^u g_{mj}(v)\big(1-\ee\big[z^{S^{\bm{Q}}_{i\leftarrow m}(u-v)}\big] \big)\mathrm{d}v \notag\\
    &\quad \hspace{8mm}- C_{mj}\Gamma(1-\gamma_{mj}) \Big(\int_0^u g_{mj}(v)\big(1-\ee\big[z^{S^{\bm{Q}}_{i\leftarrow m}(u-v)}\big] \big)\mathrm{d}v\Big)^{\gamma_{mj}} \Big\} \notag\\
    &= (1-z^{\bm{1}_{\{i=j\}}})\pp(J_i>u) + \sum_{m=1}^d  \ee[B_{mj}] \int_0^u g_{mj}(v)\big(1-\ee\big[z^{S^{\bm{Q}}_{i\leftarrow m}(u-v)}\big] \big)\mathrm{d}v \notag\\
    &\quad + \sum_{m=1}^d C_{mj}\Gamma(1-\gamma_{mj}) \Big(\int_0^u g_{mj}(v)\big(1-\ee\big[z^{S^{\bm{Q}}_{i\leftarrow m}(u-v)}\big] \big)\mathrm{d}v\Big)^{\gamma_{mj}}.
\end{align}

We now focus on the linear terms, and the next term of order strictly between $1$ and $2$, yet to be determined.
As for the linear terms, consider the expansion
\begin{align*}
    1-\ee\big[z^{S^{\bm{Q}}_{i\leftarrow j}(u)}\big]
    &= 1 - \left( 1 - (1-z)\ee\big[ S^{\bm{Q}}_{i\leftarrow j}(u)\big]\right) + O((1-z)^2) \\
    &= (1-z)\ee\big[S^{\bm{Q}}_{i\leftarrow j}(u)\big] + O((1-z)^2) = (1-z)R^{\bm{Q}}_{ij}(u) + O((1-z)^2),
\end{align*}
where the last equality follows from the observation $R_{ij}^{\bm{Q}}(u)= \ee[S^{\bm{Q}}_{i\leftarrow j}(u)]$;
applying the distributional equality \eqref{eq: distributional equality S_ij} to $\ee[S^{\bm{Q}}_{i\leftarrow j}(u)]$ yields Eqn.\ \eqref{eq: R_ij linear term equation}.

For the next term (beyond the linear terms, that is), we expand $1-\ee\big[z^{S^{\bm{Q}}_{i\leftarrow j}(u)}\big]$ again,
now including a term of fractional order $\vartheta_{ij} \in (1,2)$ (whose value will be determined below), yielding
\begin{align*}
    1-\ee\big[z^{S^{\bm{Q}}_{i\leftarrow j}(u)}\big] =  (1-z)  R^{\bm{Q}}_{ij}(u)+ (1-z)^{\vartheta_{ij}} R^{\bm{Q},\delta}_{ij}(u) + O((1-z)^2),
\end{align*}
where $R^{\bm{Q},\delta}_{ij}(u)$ is the solution to (\ref{eq: R_ij^gamma fractional term equation}), which we argue next.
By substituting this expansion in (\ref{eq: 1-eta_c with B_mj}), equating the terms of order between $1$ and $2$, and ignoring higher-order terms,
we obtain
\begin{align*}
    (1-z)^{\vartheta_{ij}}R^{\bm{Q},\delta}_{ij}(u)
    &= \sum_{m=1}^d (1-z)^{\vartheta_{im}} \ee[B_{mj}] \int_0^t g_{mj}(v)R^{\bm{Q},\delta}_{im}(u-v) \mathrm{d}v\\
    &\quad + \sum_{m=1}^d C_{mj}(1-z)^{\gamma_{mj}}\Gamma(1-\gamma_{mj}) \Big(\int_0^u g_{mj}(v)R^{\bm{Q}}_{im}(u-v)\mathrm{d}v\Big)^{\gamma_{mj}}.\notag
\end{align*}

The problem of solving this system of equations comes down to determining $\vartheta_{i1},\dots,\vartheta_{id}$
and finding the term(s) of lowest order among the $\gamma_{mj}$.
In other words, we need to equate $\vartheta_{ij}$ on the LHS with the minimal $\gamma_{mj}$ on the RHS among the non-zero terms.
As $R^{\bm{Q}}_{im}(u) = \ee[S^{\bm{Q}}_{i\leftarrow m}(u)]$,
we have by Lemma \ref{lem: equivalence R_ij > 0 and path existence} that $R^{\bm{Q}}_{im}(u)$ is non-zero if there is a path from $m$ to $i$ in the Hawkes graph.
We therefore obtain the condition $\vartheta_{ij} = \delta_{ij}$,
with $\delta_{ij}$ introduced at the beginning of Section \ref{subsec:marginaltail}.
Note that $\vartheta_{ij}$ does not necessarily equal $\gamma_{ij}$, but rather corresponds to the heaviest tail originating from $j$ with a potential path of propagation to $i$.
Taking into account the source and target component of this heaviest tail, we obtain Eqn.~(\ref{eq: R_ij^gamma fractional term equation}).
Finally, we substitute this back into the $z$-transform of $Q_i(t)$ in (\ref{eq: def z transform Q_i and fixed point}), yielding, up to $O((1-z)^2)$ terms, as $z\uparrow 1$,
\begin{align*}
    \ee\Big[ z^{Q_i(t)} \Big]
    &\sim \prod_{j=1}^d \exp\Big( -\lambda_{j,\infty} \int_0^t \big( (1-z) R^{\bm{Q}}_{ij}(u) + (1-z)^{\vartheta_{ij}} R^{\bm{Q},\delta}_{ij}(u)\big)\mathrm{d}u \Big) \\
    &\sim 1 - (1-z) \sum_{j=1}^d \lambda_{j,\infty}\int_0^t R^{\bm{Q}}_{ij}(u)\mathrm{d}u
    -  \sum_{j=1}^d (1-z)^{\vartheta_{ij}}\lambda_{j,\infty}\int_0^t R^{\bm{Q},\delta}_{ij}(u)\mathrm{d}u,
\end{align*}
where in the last asymptotic equality we used the Taylor expansion of the exponential expression.
To obtain the lowest order term among the $\vartheta_{ij}$, we set $\bar \gamma_i := \min_{j\in[d]} \{\vartheta_{ij}\}$,
and the components from where the heaviest tails originate are given by $I_i := \argmin_{j\in[d]}\{\vartheta_{ij}\} \subseteq [d]$.
After rewriting this last expression and observing that the linear term equals $\ee[Q_i(t)]$, we obtain
\begin{align}
\label{eq: z transform expansion Q_i}
    \ee\Big[ z^{Q_i(t)} \Big] - 1 + (1-z)\ee[Q_i(t)] \sim -(1-z)^{\bar
    \gamma_i}\sum_{j\in I_i}\lambda_{j,\infty}\int_0^t R^{\bm{Q},\bar \gamma_i}_{ij}(u)\mathrm{d}u.
\end{align}

Note that \eqref{eq: z transform expansion Q_i} is now in the general form stated in \cite[Theorem 8.1.6]{BGT87},
as an expansion of the $z$-transform of $Q_i(t)$, which yields that $Q_i(t) \in {\rm APT}(\bar C^{\bm{Q}}_i,\bar \gamma_i)$ for some $\bar C^{\bm{Q}}_i>0$.
\end{proof}

The result of Theorem \ref{thm: Q_i and lambda_i heavy tailed} in combination with suitable Tauberian theorems allows us to describe the asymptotic tail behavior of
$Q_i(t)$ and $\lambda_i(t)$.
In order to properly take care of the constants that appear in the application of this theorem,
we introduce, for $\star \in \{\bm{Q},\bm{\lambda}\}$,
the function $\bar R_{ij}^{\star, \delta}(\cdot)$ through
\begin{equation}
R_{ij}^{\star,\delta}(u) = \Gamma(1-\delta)\bar R_{ij}^{\star, \delta}(u).
\label{eq:barR}
\end{equation}
Note that $\bar R_{ij}^{\star, \delta}(u)$ satisfies Eqn.~\eqref{eq: R_ij^gamma fractional term equation} without the $\omega_i$ term.
We can then rewrite Eqn.~\eqref{eq: z transform expansion Q_i} as
\begin{align*}
    \ee\Big[ z^{Q_i(t)} \Big] - 1 + (1-z)\ee[Q_i(t)] \sim -(1-z)^{\bar
    \gamma_i}\Gamma(1-\bar
    \gamma_i)\sum_{j\in I_i}\lambda_{j,\infty}\int_0^t \bar R^{\bm{Q},\bar \gamma_i}_{ij}(u)\mathrm{d}u.
\end{align*}

\begin{corollary}
\label{cor: tail probability Q_i and lamda_i}
Fix $i\in[d]$ and $t\in\rr_+.$ Assume that $\bar\gamma_i\in(1,2)$.
Then, as $x\to\infty$,
\begin{align}
\label{eq: tail prob queue and lambda corollary}
\begin{split}
    {\mathbb P}(Q_i(t) > x) &\sim \left(\sum_{j \in I_i}\lambda_{j,\infty}\int_0^t \bar R^{\bm{Q},\bar\gamma_i}_{ij}(u)\,\mathrm{d}u\right)x^{-\bar\gamma_i},\\
    {\mathbb P}(\lambda_i(t) > x) &\sim \left( \sum_{j\in I_i}\lambda_{j,\infty}\int_0^t \bar R^{\bm{\lambda},\bar\gamma_i}_{ij}(u)\,\mathrm{d}u\right)x^{-\bar
    \gamma_i}.
\end{split}
\end{align}
\end{corollary}

An interesting special case of Theorem~\ref{thm: Q_i and lambda_i heavy tailed}, and Corollary~\ref{cor: tail probability Q_i and lamda_i},
concerns the situation in which the Hawkes graph is \textit{irreducible}, i.e., the case where there exists a directed path between every pair of vertices.
Theorem~\ref{thm: Q_i and lambda_i heavy tailed} then implies that all components $Q_i(t)$ and $\lambda_i(t)$, $i\in[d]$, inherit the minimum $\gamma =\min_{i,j\in[d]}\{\gamma_{ij}\}$,
i.e., all components $Q_i(t)$ and $\lambda_i(t)$ are APT with tail index $\gamma$.

\begin{remark}
In the setting of Theorem~\ref{thm: Q_i and lambda_i heavy tailed}, we have assumed $\bar\gamma_i\in(1,2)$.
In case $\bar\gamma_i\in(m,m+1)$ for some $m\in\{2,3,\ldots\}$, a similar analysis can be performed,
but additional terms are to be included in the expansions of \eqref{eq: def z transform Q_i and fixed point}.
The statement that $Q_i(t)$ and $\lambda_i(t)$ are APT with tail index $\bar\gamma_i$ carries over.
While the proof is conceptually analogous to the case of $\bar\gamma_i \in (1,2)$,
various expressions that will appear in the corresponding proof are substantially more involved.
\end{remark}

\begin{remark}
We conclude this subsection with a remark on extensions to cases in which some of the jump size distributions are not of APT type.
Following Definition \ref{def:APT}, we have assumed that the $B_{ij}$ are either of APT type or identical to $0$,
for a transparent exposition.
At the expense of additional notation and administration in the proof,
also the cases that some of the $B_{ij}$ are identical to a positive constant $b_{ij}$ can be covered.
Likewise, using a similar but more cumbersome treatment, one can also handle the situation in which, besides the dominant jump sizes $B_{ij}$ of APT type,
there are light-tailed jump sizes as well.
\end{remark}

\subsection{The tail index is a class property}
To fully appreciate how the chain structure of the Hawkes graph is reflected in the tail indices of $Q_i(t)$ and $\lambda_i(t)$, $i\in[d]$,
this subsection systematically studies this feature,
starting by revisiting the two examples discussed earlier.
We focus primarily on $Q_i(t)$; the analysis for $\lambda_i(t)$ is analogous.

In the setting of Example~\ref{example: bivariate hawkes graph}, the tail index of $Q_1(t)$ is $\gamma_{11}$
and the tail index of $Q_2(t)$ is $\bar\gamma_2 =\min \{\gamma_{21}, \gamma_{22}\}$.
In Example~\ref{example: trivariate hawkes graph}, $Q_1(t)$ has tail index $\gamma_{11}$,
whereas both $Q_2(t)$ and $Q_3(t)$ (which together form a recurrent class) have tail index $\bar\gamma_2 = \bar\gamma_3 = \min\{\gamma_{21},\gamma_{32},\gamma_{23},\gamma_{33}\}$.
The above suggests that the tail index is, like recurrence, transience and periodicity for Markov chains, a {\it class property},
i.e., the tail index is the same for all states belonging to a class.
This claim is confirmed in the next result.

\begin{proposition}
For $i\in[d]$ in a given class of the Hawkes graph, all ${\mathbb P}(Q_i(t)>x)$ have the same tail index, i.e., the tail index is a class property.
\end{proposition}

\begin{proof}
The proof follows from Theorem \ref{thm: Q_i and lambda_i heavy tailed},
in particular, from the role played by $\delta_{ij}$ in the proof of that theorem.
\end{proof}

The following somewhat more elaborate example demonstrates how in general
to iteratively compute the tail indices corresponding to the individual classes.

\begin{example}[\textsc{A more involved Hawkes graph}]
\label{example: sixvariate hawkes graph}
In this example, we consider a system of 6 states such that the vertex set $V = \{Q_1,\dots,Q_6\}$ and the directed edges $E$ are drawn below.
\begin{center}
\begin{tikzpicture}[->,thick,node distance = 3cm, roundnode/.style={circle, draw= black, thick, minimum size = 7mm}]
\node[roundnode,green] (Q1) at (0,0) {$Q_1$};
\node[roundnode,dashed, red][right of =Q1] (Q2) {$Q_2$};
\node[roundnode,dashed, red][right of =Q2] (Q3) {$Q_3$};
\begin{scope}[node distance = 2.5cm]
\node[roundnode,dash dot, cyan][below of =Q1] (Q4) {$Q_4$};
\node[roundnode,dash dot, cyan][below of= Q2] (Q5) {$Q_5$};
\node[roundnode,dotted, blue][below of= Q3] (Q6) {$Q_6$};
\end{scope}
\draw[black] (Q1) to [out=30,in=150] node[above]{$e_{21}$} (Q2);
\draw[black] (Q1) to [out=150,in = 205,looseness = 7] node[left]{$e_{11}$}  (Q1);
\draw[black] (Q2) to [out=30,in=150] node[above]{$e_{32}$} (Q3);
\draw[black] (Q3) to [out=30,in = 330,looseness = 7] node[right]{$e_{33}$} (Q3);
\draw[black] (Q3) to [out=210,in=330] node[below]{$e_{23}$} (Q2);
\draw[black] (Q4) to [out=150, in=225] node[left]{$e_{14}$} (Q1);
\draw[black] (Q4) to [out=30, in = 150] node[above]{$e_{54}$} (Q5);
\draw[black] (Q5) to [out=210, in=330] node[below]{$e_{45}$} (Q4);
\draw[black] (Q6) to [out=150, in = 225] node[right]{$e_{36}$} (Q3);
\draw[black] (Q6) to [out=30,in = 330,looseness = 7] node[right]{$e_{66}$} (Q6);
\end{tikzpicture}
\end{center}
The different colors (drawing styles) represent the class each vertex belongs to.
Observe that there are three transient classes, viz.\ ${\rm green~(G, solid)}$, {\rm cyan~(C, dashed}-{\rm dotted)} and ${\rm blue~(B, dotted)}$,
and one recurrent class, viz.\ ${\rm red~(R, dashed)}$.
We start by determining the tail index of the transient class that is the `furthest away' from the recurrent class, viz.\ ${\rm C}$.
This class has distance $2$ to the recurrent class, as one has to go through another transient class to reach the recurrent class.
In evident notation,
\[\gamma_{\rm C} = \bar\gamma_4 = \bar\gamma_5 = \min\{\gamma_{45},\gamma_{54}\}.\]
We proceed with the transient classes with distance $1$ to the recurrent class, i.e., ${\rm G}$ and ${\rm  B}$, and find
\begin{align*}
    \gamma_{\rm G}&=\bar\gamma_1 =
    \min\{\gamma_{\rm C},\gamma_{14}, \gamma_{11}\}
    =\min\{\gamma_{54},\gamma_{45},\gamma_{14}, \gamma_{11}\}, \\
    \gamma_{\rm B}&=\bar\gamma_6 = \gamma_{66}.
\end{align*}
We finally determine the tail indices of the recurrent classes.
In this case, there is just one recurrent class, viz.\  ${\rm R}$:
\begin{align*}
   \gamma_{\rm R}&= \bar\gamma_2 = \bar\gamma_3 =
    \min\{ \gamma_{\rm G},  \gamma_{\rm B},\gamma_{21},\gamma_{36},\gamma_{32},\gamma_{23},\gamma_{33}\}\\
    &=\min\{\gamma_{45},\gamma_{54},\gamma_{14},\gamma_{11},\gamma_{66}, \gamma_{21},\gamma_{36},\gamma_{32},\gamma_{23},\gamma_{33}\}.
\end{align*}
\end{example}

\subsection{The asymptotic behavior of linear combinations}
In the above, we have focused on the asymptotic (marginal) tail behavior of the $i$-th components $Q_i(t)$ and $\lambda_i(t)$.
To gain insight into the corresponding joint asymptotic behavior, we now consider the tail behavior of $\langle \bm{c},\bm{Q}(t)\rangle$ for some ${\bm{c}} \in \rr_+^d$,
where $\langle\bm{x},\bm{y}\rangle$ denotes the inner product.
As before, the case $\langle \bm{c},\bm{\lambda}(t)\rangle$ can be dealt with analogously.
For convenience, we do so for the case that the Hawkes graph is irreducible;
the case that the Hawkes graph is not irreducible can be addressed as well, but its analysis is somewhat cumbersome and mechanical,
as it requires distinguishing various cases.

As demonstrated above, in the irreducible case, $\bar\gamma_i=\gamma$ for all $i\in[d]$.
Likewise, also the set $I_i$ does not depend on $i$, and is therefore denoted simply by $I$.

\begin{proposition}
\label{cor: heavy tail linear combination queue}
Fix $t\in\rr_+$ and $\bm{c} \in \rr_+^d$.
Let the Hawkes graph be irreducible.
Assume $\gamma\in(1,2)$.
Then, $\langle \bm{c},\bm{Q}(t)\rangle \in {\rm APT}(\hat C, \gamma)$ for some $\hat C>0$.
More precisely,
\begin{align}
\label{eq: lin cominbation queue heavy tailed thm statement}
\begin{split}
    &{\mathbb E}\Big[z^{\langle \bm{c},\bm{Q}(t)\rangle}\Big] - 1 + (1-z) {\mathbb E}\big[\langle \bm{c},\bm{Q}(t)\rangle\big]
    \sim -(1-z)^{\gamma}\sum_{i=1}^d c_i^\gamma \sum_{j\in I}\lambda_{j,\infty}\int_0^t R^{\bm{Q},\gamma}_{ij}(u)\,\mathrm{d}u,
\end{split}
\end{align}
as $z\uparrow 1$.
The first moment equals
\begin{equation*}
{\mathbb E}\big[\langle \bm{c},\bm{Q}(t)\rangle\big] = \sum_{j=1}^d \lambda_{j,\infty} \int_0^t \sum_{i=1}^d c_i R^{\bm{Q}}_{ij}(u)\,\mathrm{d}u.
\end{equation*}
\end{proposition}
\begin{proof}
The proof follows the lines of the proof of Theorem \ref{thm: Q_i and lambda_i heavy tailed}; we therefore omit the details.
One has to properly take  care of the multivariate setting when taking expansions and equating linear and fractional order terms.
Due to irreducibility, the term of order $\gamma$ effectively propagates throughout the system.
\end{proof}

Proposition \ref{cor: heavy tail linear combination queue} can be used to determine the asymptotic tail behavior of $\langle \bm{c},\bm{Q}(t)\rangle$ analogous to Corollary \ref{cor: tail probability Q_i and lamda_i}:
under the assumptions of Proposition \ref{cor: heavy tail linear combination queue}, as $x \to \infty$, we have that
\begin{align}
\label{eq: tail prob linear combination queue}
    {\mathbb P}( \langle \bm{c},\bm{Q}(t)\rangle > x) \sim \left(\sum_{i=1}^d c_i^\gamma \sum_{j\in I}\lambda_{j,\infty}\int_0^t \bar R^{\bm{Q},\gamma}_{ij}(u)\,\mathrm{d}u\right)x^{-\gamma}.
\end{align}

\section{Numerical Examples}\label{sec:Numerics}

This section provides a collection of numerical examples to illustrate the results derived in the previous sections.
All the numerical computations in this section are conducted in {\tt Python} and the computer code is available from the authors upon request.
The first part of this section focuses on the exact analysis.
The characterization of the joint transform in Theorem~\ref{thm: joint transform zeta characterization}
and the fixed-point representation and convergence results in Theorems~\ref{thm: fixed point theorem}--\ref{thm: convergence of phi iteration}
enable us to numerically compute arbitrary joint moments of $N_i(t)$, $Q_i(t)$ and $\lambda_i(t)$, for any $i\in[d]$ and $t\in\rr_+$,
using standard numerical techniques.
The second part of this section considers the asymptotic analysis.
The characterization of the heavy-tailed asymptotic behavior of $Q_i(t)$ and $\lambda_i(t)$ in Theorem~\ref{thm: Q_i and lambda_i heavy tailed}
allows us to numerically evaluate their tail probabilities.
For both the exact and asymptotic analyses, we compare our numerically evaluated results to Monte Carlo simulated counterparts.
Our simulation procedure is based on Ogata's \cite{O81} thinning algorithm (see also \cite[Algorithm 1.21]{L09} for details),
which in our general multivariate setting essentially relies on the cluster representation in Definition \ref{def: hawkes cluster}.
Once a sample path of $\bm{N}(t)$ has been simulated for some $t\in\rr_+$, one can easily obtain the corresponding $\bm{\lambda}(t)$ and $\bm{Q}(t)$.
The simulation procedure is, in principle, elementary, but can be highly time consuming when $t\in\rr_+$ is large, in the presence of heavy tails,
or when a high level of precision is pursued.
By contrast, the numerical evaluation of our exact and asymptotic results is nearly instantaneous.

Because our results allow for general decay functions $g_{ij}(\cdot)$, we illustrate two examples of parametrizations, namely exponential and power-law decay.
Throughout this section, we consider the bivariate case $d=2$.
In the instances considered, we let the stability condition be satisfied,
i.e., the spectral radius of the matrix $\lVert \bm{H} \rVert = ( \lVert h_{ij}\rVert)_{i,j\in[2]}$ is smaller than $1$, which results in the condition
\begin{align}
\label{eq: stability condition bivariate general}
     (1-\norm{h_{11}})(1-\norm{h_{22}}) > \norm{h_{12}}\norm{h_{21}}.
\end{align}
The stability condition (\ref{eq: stability condition bivariate general}) implies that the processes $N_i(t)$, $Q_i(t)$ and $\lambda_i(t)$ converge to a steady state as $t$ grows.

%

\subsection{Exact analysis}
We focus attention on the processes $\bm{Q}(t)=(Q_1(t),Q_2(t))$ and $\bm{\lambda}(t) = (\lambda_1(t),\lambda_2(t))$
and illustrate the joint transform characterization of $(\bm{Q}(t),\bm{\lambda}(t))$ under both exponential and power-law decay.
More specifically, by means of a standard finite difference method, we approximate the derivatives of the joint transform $\cJ_{\bm{Q},\bm{\lambda}}(t)$
obtained by relying on the fixed-point characterization of Theorems~\ref{thm: joint transform zeta characterization}--\ref{thm: fixed point theorem},
so as to numerically evaluate arbitrary joint moments. 
We focus on moments corresponding to the first components, i.e.,
\[\ee[Q_1(t)] =\ee[Q_1(t)\, |\, Q_1(0) = 0],\:\:\:\:\:\ee[\lambda_1(t)] = \ee[\lambda_1(t) \, | \, \lambda_1(0) = \lambda_{1,\infty}],\]
and the associated variances, as well as the joint moments $\ee[Q_1(t)Q_2(t)]$ and $\ee[Q_1(t)\lambda_1(t)]$.
Recall that the departures of events in $Q_i(t)$ are governed by the non-negative random variable $J_i$,
which we assume to be exponentially distributed with parameter $\mu_i > 0$.

\subsubsection{Exponential}
We let the decay functions $g_{ij}(\cdot)$, with $i,j\in[2]$, be of exponential form.
In our bivariate setting, we assume
\begin{align}
\label{eq: numerics - exp decay bivariate}
    g_{11}(t) = g_{12}(t) = e^{-\alpha_1 t}, \hspace{1.2cm} g_{21}(t)= g_{22}(t) = e^{-\alpha_2 t},
\end{align}
for $\alpha_1,\alpha_2 >0$.
Further, in this example we take the random variables $B_{ij}$ to be positive constants, i.e., $B_{ij} \equiv b_{ij} \in \rr_+$.
This yields, for $i,j\in[2]$, that $\norm{h_{ij}} = b_{ij}/\alpha_i$,
and we will select parameters such that the bivariate stability condition in Eqn.\ (\ref{eq: stability condition bivariate general}) is satisfied.

\begin{figure}
\includegraphics[width=0.49\textwidth]{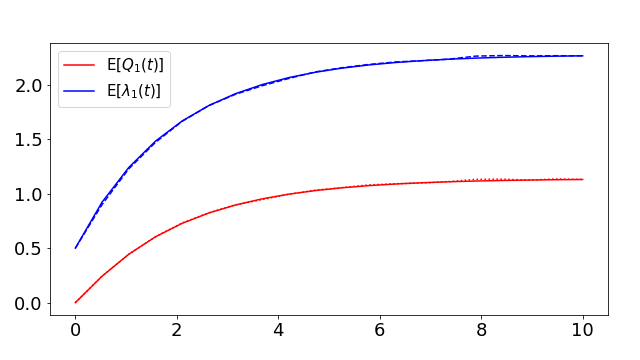}
\includegraphics[width=0.49\textwidth]{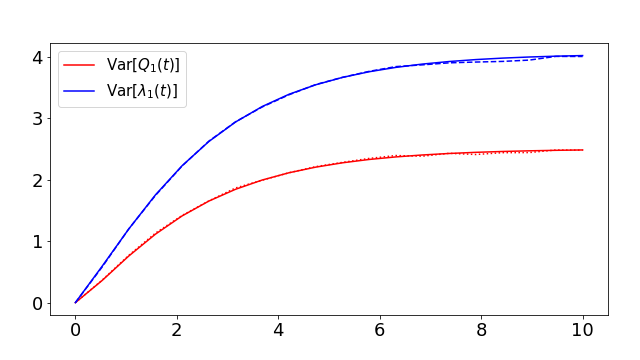}
\includegraphics[width=0.49\textwidth]{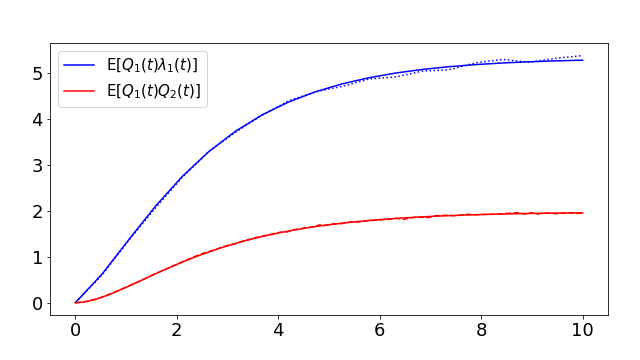}
\caption{\small \textit{Plots of the expectations and variances of $Q_1(\cdot)$ and $\lambda_1(\cdot)$ and the joint moments $\ee[Q_1(t)\lambda_1(t)]$ and $\ee[Q_1(t)Q_2(t)]$ (solid lines),
compared to Monte Carlo simulated averages (dashed lines), in the bivariate model $(d=2)$ under exponential decay, with $t\in[0,10]$.
Parameters: $\lambda_{1,\infty} = \lambda_{2,\infty} = 0.5$, $\mu_1 = \mu_2 =2$,
$\alpha_{11}=\alpha_{12} = 2.3$, $\alpha_{21} = \alpha_{22} = 2$, $B_{11} \equiv 1.3$, $B_{12} \equiv 0.6$, $B_{21} \equiv 0.8$, $B_{22} \equiv 0.5$.}}
\label{fig: exponential exact moments}
\end{figure}

In Figure~\ref{fig: exponential exact moments}, we plot the expectations $\ee[Q_1(t)]$ and $\ee[\lambda_1(t)]$,
the variances $\text{Var}[Q_1(t)]$ and $\text{Var}[\lambda_1(t)]$,
and the joint moments $\ee[Q_1(t)\lambda_1(t)]$ and $\ee[Q_1(t)Q_2(t)]$,
obtained from the fixed-point characterization.
These are plotted against their Monte Carlo simulated counterparts based on $M=10^5$ runs, for $t\in[0,10]$.
In all cases, the numerical evaluation of our exact results closely matches the Monte Carlo simulated counterparts.
Around $t=6$, the effect of the initial state has vanished, in that the processes enter the stationary regime.
We note that the CPU time associated with the exact results is negligible compared to the CPU time required for the simulation results.
Furthermore, for large $t$, the exact results remain smooth while the simulation results become a little rougher.

\subsubsection{Power}
We next let the decay functions $g_{ij}(\cdot)$, with $i,j\in[2]$, be of power-law type.
In particular, we take
\begin{align}
\label{eq: numerics - power law decay bivariate}
    g_{11}(t) = g_{12}(t) = \frac{1}{(c_1 + t)^{p_1}}, \hspace{1.2cm} g_{21}(t)= g_{22}(t) =  \frac{1}{(c_2 + t)^{p_2}},
\end{align}
where $c_1,c_2 >0$ and $p_1,p_2>1$ to ensure integrability.
Again taking $B_{ij} \equiv b_{ij} \in \rr_+$, we now have $\norm{h_{ij}} = b_{ij}c_i^{1-p_i}(p_i -1)^{-1}$.
As before, we choose the parameters such that the bivariate stability condition in Eqn.\ (\ref{eq: stability condition bivariate general}) is satisfied.

\begin{figure}
\includegraphics[width=0.49\textwidth]{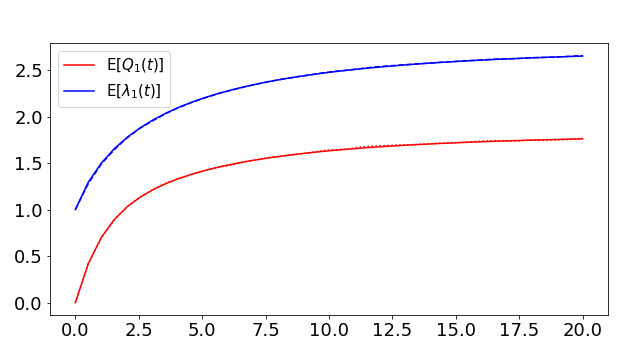}
\includegraphics[width=0.49\textwidth]{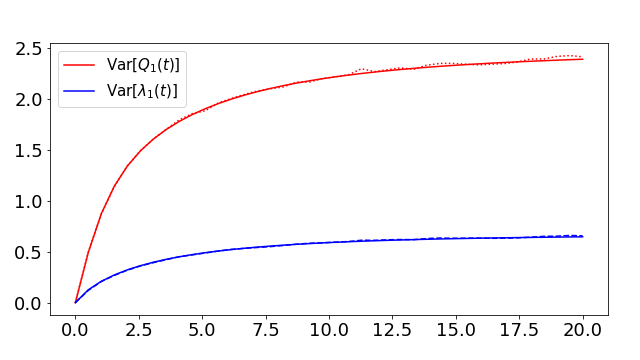}
\includegraphics[width=0.49\textwidth]{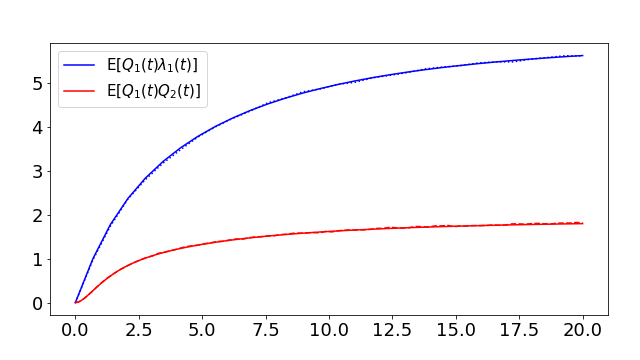}
\caption{\small \textit{Plots of the expectations and variances of $Q_1(\cdot)$ and $\lambda_1(\cdot)$ and the joint moments $\ee[Q_1(t)\lambda_1(t)]$ and $\ee[Q_1(t)Q_2(t)]$ (solid lines),
compared to Monte Carlo simulated averages (dashed lines), in the bivariate model $(d=2)$ under power-law decay, with $t\in[0,20]$.
Parameters: $\lambda_{1,\infty} = \lambda_{2,\infty} = 1$, $\mu_1 = \mu_2 =1.5$, $c_1 = 1.5$, $c_2 = 2$, $p_1 = 2.5$, $p_2 = 3$, $B_{11} \equiv 1.5$, $B_{12} \equiv 0.5$, $B_{21} \equiv 1$, $B_{22} \equiv 0.5$.}}
\label{fig: PL exact moments}
\end{figure}

Figure~\ref{fig: PL exact moments} displays the expectations $\ee[Q_1(t)]$ and $\ee[\lambda_1(t)]$,
the variances $\text{Var}[Q_1(t)]$ and $\text{Var}[\lambda_1(t)]$,
as well as the joint moments $\ee[Q_1(t)\lambda_1(t)]$ and $\ee[Q_1(t)Q_2(t)]$,
plotted against their Monte Carlo simulated counterparts.
We observe again a highly accurate match.
We note that, compared to the case of exponential decay, it now takes longer for the processes to reach the stationary regime.
This is due to the fatter tails of power-law decay: heuristically, the processes now `have more memory'.

\subsection{Asymptotic analysis}
We proceed by numerically illustrating our asymptotic results on the tail probabilities of $N_i(t)$, 
as established in Corollary~\ref{cor: tail probability Q_i and lamda_i}.
In particular, we compute the $\bar R_{ij}^{\bm{N},\delta}(\cdot)$,
function appearing in Corollary~\ref{cor: tail probability Q_i and lamda_i} and the corresponding tail probability approximations,
and compare these to tail probabilities estimated using Monte Carlo simulation.

In our specific bivariate example, we consider only one direction of cross-excitation and one heavy-tailed random variable.
More precisely, we set $B_{11} \equiv 0$, $B_{21} \equiv 0$, $B_{12} = 1$ and assume $B_{22} \in \text{APT}(1,\gamma)$ with $\gamma=1.8$,
such that this system can be represented by the following Hawkes graph:
\begin{center}
\begin{tikzpicture}[->,thick,node distance = 3cm, roundnode/.style={circle, draw= black, thick, minimum size = 7mm}]
\node[roundnode] (N1) at (0,0) {$N_1$};
\node[roundnode][right of =Q1] (N2) {$N_2$};
\draw[black] (N2) to [out=150,in=30] node[above]{$e_{12}$} (N1);
\draw[black] (N2) to [out=30,in = 330,looseness = 7] node[right]{$e_{22}$} (N2);
\end{tikzpicture}
\end{center}
The index $\gamma$ of the heavy-tailed random variable $B_{22}$ is inherited by both $N_1(\cdot)$ and $N_2(\cdot)$.
In this setting, we compute the functions $R_{ij}^{\bm{N}}(\cdot)$ and $\bar{R}_{ij}^{{\bm{N}},\gamma}(\cdot)$, with $i,j\in[2]$,
using the system of equations of their Laplace-Stieltjes transforms given in Eqns.~\eqref{eq: Laplace-Stieltjes transform equations R_ij^Q},
\eqref{eq: Laplace-Stieltjes transform equations R_ij^Q,gamma} and \eqref{eq:barR}.
By solving this linear system and inverting back, we obtain $R_{ij}^{\bm{N}}(\cdot)$ at time $u$ by
\begin{align}
\label{eq: bivariate system of R_ij^N equations}
\begin{split}
    R_{12}^{\bm{N}}(u) &= \ee[B_{12}] \cL^{-1}\Big\{ \frac{\cL\{G_{12}\}(\cdot)}{1-\ee[B_{22}]\cL\{g_{22}\}(\cdot)}\Big\}(u), \\
    R_{22}^{\bm{N}}(u) &=\cL^{-1} \Big\{\frac{\cL\{1\}(\cdot)}{1-\ee[B_{22}]\cL\{g_{22}\}(\cdot)}\Big\}(u),
\end{split}
\end{align}
where $G_{12}(u)=\int_0^u g_{12}(v)\mathrm{d}v$.
Note that $R^{\bm{N}}_{11}(\cdot) \equiv 1$ by Eqn.\ (\ref{eq: R_ij linear term equation})
since $B_{11} \equiv 0$, and $R^{\bm{N}}_{21}(\cdot) \equiv 0$ by Lemma \ref{lem: equivalence R_ij > 0 and path existence} as $P_{2\leftarrow 1} =0$.
In a similar manner, we obtain $\bar R_{ij}^{{\bm{N}},\gamma}(\cdot)$ at time $u$ by
\begin{align}
\label{eq: bivariate system of R_ij^N gamma equations}
\begin{split}
    \bar R_{12}^{\bm{N},\gamma}(u) &= \cL^{-1} \Big\{ \frac{\cL\big\{ (g_{22} * \bar R_{12}^{\bm{N}})^{\gamma}\big\}(\cdot)}{1-\ee[B_{22}]\cL\{g_{22}\}(\cdot)}\Big\}(u),\\
    \bar R_{22}^{\bm{N},\gamma}(u) &= \cL^{-1} \Big\{ \frac{\cL\big\{ (g_{22} * \bar R_{22}^{\bm{N}})^{\gamma}\big\}(\cdot)}{1-\ee[B_{22}]\cL\{g_{22}\}(\cdot)}\Big\}(u).
\end{split}
\end{align}
This allows us to compute the analytical expressions that appear in Corollary~\ref{cor: tail probability Q_i and lamda_i},
which in our bivariate setting are $\int_0^t \bar R_{12}^{\bm{N},\gamma}(u)\mathrm{d}u$ and $\int_0^t \bar R_{22}^{\bm{N},\gamma}(u)\mathrm{d}u$
for components $N_1(t)$ and $N_2(t)$, respectively.
We next discuss two parametrizations of the decay functions to compute these terms explicitly.

\subsubsection{Exponential}
We choose our decay functions as in Eqn.\ (\ref{eq: numerics - exp decay bivariate}),
and select the remaining parameters such that the stability condition (\ref{eq: stability condition bivariate general}) is satisfied.
In Figure~\ref{fig: exponential tail probabilities}, we plot the analytical expressions from Corollary~\ref{cor: tail probability Q_i and lamda_i}
against the Monte Carlo simulation-based approximations of the tail probabilities $\pp(N_1(t)>x)$ and $\pp(N_2(t)>x)$ as $x$ grows, for fixed $t=1$.
The simulations are performed by sampling $M=2\cdot10^6$ runs of $\bm{N}(1)=(N_1(1),N_2(1))$ and counting the proportion of these runs that lead to values larger than any given threshold $x$.
The number of simulation runs is chosen sufficiently large to obtain reasonable estimates of small tail probabilities.
As expected, we see that as $x$ grows, the simulation approximations of the tail probabilities converge toward the analytical asymptotic expressions.

\begin{figure}
\includegraphics[scale=.5]{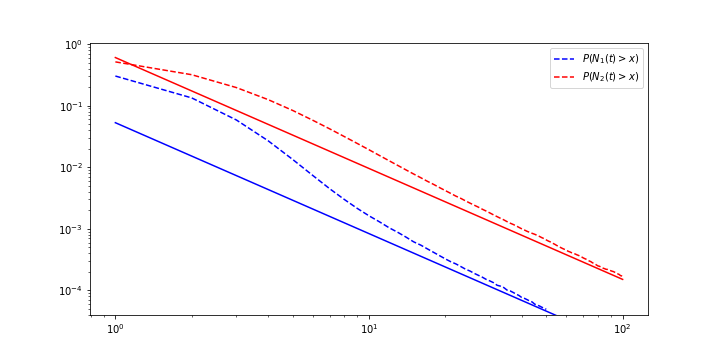}
\caption{\small \textit{Plots of the tail probabilities $\pp(N_1(t)>x)$ and $\pp(N_2(t)>x)$ for fixed $t=1$ in the bivariate model ($d=2$)
under exponential decay, using the analytical expressions appearing in Corollary~\ref{cor: tail probability Q_i and lamda_i} (solid lines),
compared to Monte Carlo simulated approximations (dashed lines).
Parameters: $\lambda_{1,\infty} = 0.5$, $\lambda_{2,\infty} = 1.5$, $\alpha_{11}=\alpha_{12}= \alpha_{21} = \alpha_{22} = 1.5$, $B_{11} \equiv 0$, $B_{12} \equiv 1$, $B_{21} \equiv 0$, $B_{22}\in\text{\rm APT}(1,1.8)$.}}
\label{fig: exponential tail probabilities}
\end{figure}

\subsubsection{Power}
We now choose our decay functions as in Eqn.\ \eqref{eq: numerics - power law decay bivariate},
and the remaining parameters such that the stability condition \eqref{eq: stability condition bivariate general} is satisfied.
We substitute the decay functions into \eqref{eq: bivariate system of R_ij^N equations} and \eqref{eq: bivariate system of R_ij^N gamma equations},
and compute the analytical expressions appearing in Corollary~\ref{cor: tail probability Q_i and lamda_i}.
Figure~\ref{fig: power law tail probabilities} is the counterpart of Figure~\ref{fig: exponential tail probabilities}, but now for power-law decay.
We observe again that both Monte Carlo simulated approximations converge to the analytical asymptotic expressions.

\begin{figure}
\includegraphics[scale=.5]{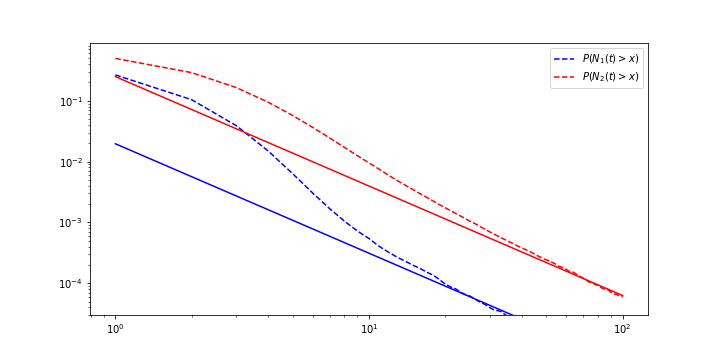}
\caption{\small \textit{Plots of the tail probabilities $\pp(N_1(t)>x)$ and $\pp(N_2(t)>x)$ for fixed $t=1$ in the bivariate model ($d=2$)
under power-law decay, using the analytical expressions appearing in Corollary~\ref{cor: tail probability Q_i and lamda_i} (solid lines),
compared to Monte Carlo simulated approximations (dashed lines).
Parameters: $\lambda_{1,\infty} = 0.5$, $\lambda_{2,\infty} = 1.5$, $c_1=c_2=1$, $p_1=2.5$, $p_2=3.5$, $B_{11} \equiv 0$, $B_{12} \equiv 1$, $B_{21} \equiv 0$, $B_{22}\in\text{\rm APT}(1,1.8)$.}}
\label{fig: power law tail probabilities}
\end{figure}

\section{Conclusions}\label{sec:Conclusion}

This paper has focused on multivariate Hawkes processes and associated population processes, establishing both exact and asymptotic results.
Importantly, we allow our model to be non-Markovian, in that the processes' decay functions can be chosen generally.
We have characterized their joint transform via a fixed-point theorem, by exploiting suitable extensions of the existing results on cluster representations for Hawkes processes and associated distributional equalities induced by the underlying branching structure.
For the case that the intensity jumps are heavy-tailed, we have also succeeded in determining the corresponding asymptotic tail behavior.

A (partially) more general setting than that analyzed in our paper is provided by the extensions of Hawkes processes introduced in \cite{BM96},
referred to as \textit{non-linear} Hawkes processes, and versatilely analyzed further in \cite{T16,Z13,Z15}.
Then, the form of the intensity process is governed by a general auxiliary function, rather than the additive structure~\eqref{eq: intensity lambda def}.
Importantly, the cluster representation applies to the additive case only.
As a consequence, our results cannot be readily generalized to the non-linear case, which will require an intrinsically different approach.

\newpage

\appendix

\section{Additional Proofs}
\label{section: proofs}

\begin{proof}[Proof of Lemma~\ref{lem: phi well defined}]
We start by considering a pair of matrix-valued processes $(\bm{X}(\cdot),\bm{Y}(\cdot)) = \big(X_{ij}(\cdot),Y_{ij}(\cdot)\big)_{i,j\in[d]}$
and denote by $\bm{X}_{(i)}(\cdot)$ and $\bm{Y}_{(i)}(\cdot)$ the $i$-th rows of the respective matrices.
We define a mapping $\Phi$ acting on the space of such pairs of processes at time $u$ by
\begin{align}
\label{eq: appendix mapping X,Y to distributional eq A(X,Y)}
    \big(X_{ij}(u),Y_{ij}(u)\big)_{i,j\in[d]}
    &\mapsto \Phi\big((X_{ij}(\cdot),Y_{ij}(\cdot))_{i,j\in[d]}\big)(u) \equiv \Phi\big(\bm{X},\bm{Y}\big)(u),\notag \\
    &=: \big(\cA_j\{\bm{1}_{\{i=j\}}\bm{1}_{\{J_i>u\}},\bm{X}_{(i)}(\cdot)\}(u),
    \, \cA_j\{B_{ij}g_{ij}(u),\bm{Y}_{(i)}(\cdot)\}(u)\big)_{i,j\in[d]},
\end{align}
where $\cA_j$ is as defined in Eqn.\ (\ref{eq: def linear functional A_ij branching structure first gen}).
We note that the above mapping could be equivalently expressed in terms of distribution functions of $\bm{X}(\cdot)$ and $\bm{Y}(\cdot)$. 
Now observe that the mapping in Eqn.\ (\ref{eq: appendix mapping X,Y to distributional eq A(X,Y)})
is the probabilistic analogue of the mapping $\phi$ given in Eqn.\ (\ref{eq: phi_j definition mapping}).
We can verify this, as the time-dependent joint transform $\bm{\cJ}_{\Phi(\bm{X},\bm{Y})}(\cdot)$ is, at time $u$, due to Definition \ref{def: time dependent joint transform space},
explicitly given by
\begin{align}
    \bm{\cJ}_{\Phi(\bm{X},\bm{Y})}(u)
    = \begin{bmatrix}
    \cJ_{\Phi_1(\bm{X},\bm{Y})}(u) \\
    \vdots \\
    \cJ_{\Phi_d(\bm{X},\bm{Y)}}(u)
    \end{bmatrix},
\end{align}
where each entry, for $j\in[d]$ corresponding to the $j$-th column $\Phi_j(\bm{X},\bm{Y})(\cdot)$, is given by
\begin{align}
\begin{split}
\label{eq: appendix joint transform A(X,Y)}
    \cJ_{\Phi_j(\bm{X},\bm{Y})}(u)
    &\equiv \cJ_{\Phi_j(\bm{X},\bm{Y})}(u,\bm{s},\bm{z}) \\
    &=\ee\Big[ \prod_{i=1}^d z_i^{\cA_j\{\bm{1}_{\{i=j\}}\bm{1}_{\{J_i>u\}},\,\bm{X}_{(i)}(\cdot)\}(u)}
    e^{-s_i\cA_j\{B_{ij}g_{ij}(u),\,\bm{Y}_{(i)}(\cdot)\}(u)}\Big].
\end{split}
\end{align}
Mimicking the steps of the proof of Theorem~\ref{thm: fixed point theorem}, we derive that the right-hand side of Eqn.\ (\ref{eq: appendix joint transform A(X,Y)}) equals $\phi_j(\bm{\cJ}_{\bm{X},\bm{Y}})(u)$.
Therefore, the joint transform of $\Phi_j(\bm{X},\bm{Y})(\cdot)$ satisfies $\cJ_{\Phi_j(\bm{X},\bm{Y})}(\cdot) = \phi_j(\bm{\cJ}_{\bm{X},\bm{Y}})(\cdot) \in \jj$.
Since this holds for every entry $j\in[d]$, we have that $\phi(\bm{\cJ}_{\bm{X},\bm{Y}})(\cdot) \in \jj^d$. 
\end{proof}

\begin{proof}[Proof of Lemma~\ref{lem: phi continuous and bounds}]
Let $\bm{\cJ}(\cdot),\widetilde{\bm{\cJ}}(\cdot) \in \jj^d$ and $\epsilon > 0$. We will show that for a certain choice of $\delta > 0$, we have
\begin{align*}
    \jjdnorm{\bm{\cJ} - \widetilde{\bm{\cJ}}} < \delta \implies \jjdnorm{\phi(\bm{\cJ})-\phi(\widetilde{\bm{\cJ}})} < \epsilon.
\end{align*}
It suffices to prove continuity in each entry separately.
Considering the $j$-th entry of $\phi$, which is the mapping $\phi_j$ defined in (\ref{eq: phi_j definition mapping}), observe that it can be rewritten as
\begin{align*}
    &\phi_j(\cJ_1,\dots,\cJ_d)(u,\bm{s},\bm{z}) =\ee\Big[z_j^{\bm{1}_{\{J_j>u\}}}\Big]   \\
    &\quad \times \prod_{m=1}^d \ee\Big[ \exp\Big(-B_{mj}\big( s_m g_{mj}(u) +  \int_0^u g_{mj}(v)\big(1- \cJ_m(u-v,\bm{s},\bm{z}))\mathrm{d}v\big)\Big)\Big].
\end{align*}
We then have
\begin{align}
\label{eq: phi first bound}
    &\jjnorm{\phi_j(\bm{\cJ}) - \phi_j(\widetilde{\bm{\cJ}})}^2\\
    &=\sup_{u,\bm{s},\bm{z}} |\phi_j(\bm{\cJ})(u,\bm{s},\bm{z}) - \phi_j(\widetilde{\bm{\cJ}})(u,\bm{s},\bm{z})|^2  \notag\\
    \begin{split}
    &\leqslant \sup_{u,\bm{s},\bm{z}}\Big| \ee\Big[ \exp\Big(-\sum_{m=1}^d B_{mj}\big( s_m g_{mj}(u) + \int_0^u g_{mj}(v)\big( \cJ_m(u-v,\bm{s},\bm{z})-1)\mathrm{d}v\big)\Big) \\
    &\quad \quad - \exp\Big(-\sum_{m=1}^dB_{mj} \big(s_m g_{mj}(u) + \int_0^u g_{mj}(v)(\widetilde{\cJ}_m(u-v,\bm{s},\bm{z}) - 1)\mathrm{d}v\big)\Big)\Big] \Big|^2, \label{eq: first bound}
    \end{split}
\end{align}
where we have used that $|z_j|\leqslant 1$.
We then apply the mean value theorem to the difference of the exponential terms; for an exponential function, it states that $e^a - e^b = (a-b)e^c$ for some $c\in[a,b]$.
We have that
\begin{align*}
    -\sum_{m=1}^d{B_{mj}\big(s_m g_{mj}(u) + \int_0^u  g_{mj}(v)\big(1-\cJ_m(u-v,\bm{s},\bm{z})\big)\mathrm{d}v\big)}\leqslant 0,
\end{align*}
due to $\cJ_m(u-v,\bm{s},\bm{z})\leqslant 1$, $B_{mj}$ being non negative, $\bm{s} \in \rr_+^d$ and $g_{mj}(v) \geqslant 0$ by assumption. The same holds for the terms involving $\widetilde{\cJ}_m$.
Hence, we can apply the mean value theorem with some $c\leqslant 0$.
We thus obtain the following upper bound on \eqref{eq: first bound}:
\begin{align}\label{eq:UB2}
    \sup_{u,\bm{s},\bm{z}}\Big| \sum_{m=1}^d\ee \Big[ B_{mj}\int_0^u g_{mj}(v)(\cJ_m(u-v,\bm{s},\bm{z}) - \widetilde{\cJ}_m(u-v,\bm{s},\bm{z}) )\mathrm{d}v \Big] \Big|^2,
\end{align}
since the $B_{mj}s_mg_{mj}(u)$ terms and the constants cancel. By an application of the triangle inequality we can bound this expression further. Indeed, \eqref{eq:UB2} is dominated by
\begin{align}\nonumber
    \lefteqn{ \sum_{m=1}^d \sup_{u,\bm{s},\bm{z}} \ee\Big| B_{mj}\int_0^u g_{mj}(v) \big(\cJ_m(u-v,\bm{s},\bm{z}) - \widetilde{\cJ}_m(u-v,\bm{s},\bm{z})\big) \mathrm{d}v \Big|^2 }\\
    &\leqslant \sum_{m=1}^d\ee[B_{mj}]^2\sup_{u,\bm{s},\bm{z}} \Big|\int_0^u g_{mj}(v)\big( \cJ_m(u-v,\bm{s},\bm{z}) - \widetilde{\cJ}_m(u-v,\bm{s},\bm{z})\big)\mathrm{d}v\Big|^2,\label{form}
\end{align}
where $\ee[B_{mj}]^2 < \infty$ by assumption. Finally, this term can be bounded by applying Young's inequality for convolutions, which yields that \eqref{form} is dominated by
\begin{align*}
    \lefteqn{\sum_{m=1}^d \ee[B_{mj}]^2 \sup_{u,\bm{s},\bm{z}} \Big|\int_0^u g_{mj}(v)\mathrm{d}s \int_0^u \big(\cJ_m(v,\bm{s},\bm{z}) - \widetilde{\cJ}_m(v,\bm{s},\bm{z})\big)\mathrm{d}v \Big|^2}\\
    &\leqslant d  \max_{m,j\in[d]}\ee[B_{mj}]^2 \Big|\int_0^t g_{mj}(v)\mathrm{d}v\Big|^2  \sup_{u,\bm{s},\bm{z}}\Big|\int_0^u\big(\cJ_m(v,\bm{s},\bm{z}) - \widetilde{\cJ}_m(v,\bm{s},\bm{z})\big)\mathrm{d}v\Big|^2 \\
    &\leqslant d  \max_{m,j\in[d]}\ee[B_{mj}]^2\lVert g_{mj} \rVert_{L^1(\rr_+)}^2 t \delta^{2},
\end{align*}
where $\lVert g_{mj} \rVert_{L^1(\rr_+)}^2 = (\int_0^\infty g_{mj}(v)]\mathrm{d}v)^2 < \infty$ and $|\cJ_m(v,\bm{s},\bm{z}) - \widetilde{\cJ}_m(v,\bm{s},\bm{z})| < \delta$ by assumption, in combination with the standard inequality $\sum_{m=1}^d a_i \leqslant d\max_{i\in[d]}\{a_i\}$ for real numbers $a_i$. Hence, choosing $\delta$ as
\begin{align*}
    \delta^2 = \frac{\epsilon}{td^2\max\limits_{m,j\in[d]}\ee[B_{mj}]^2 \lVert g_{mj} \rVert_{L^1(\rr_+)}^2},
\end{align*}
implies that (\ref{eq: phi first bound}) becomes $\jjnorm{\phi_j(\bm{\cJ}) - \phi_j(\tilde{\bm{\cJ}})}^2 \leqslant \epsilon^2$.
Doing this for all $j\in[d]$ yields the result.

Of course, for the proof of Lemma~\ref{lem: phi continuous and bounds} to work, we need $t>0$, $\ee[B_{mj}] > 0$ and $g_{mj} \neq 0$ for at least some combination of $m,j\in[d]$.
However, it is clear that choosing all of these to be $0$ yields an irrelevant model.
\end{proof}

\begin{proof}[Proof of Lemma~\ref{lem: equivalence R_ij > 0 and path existence}]
$(i) \Rightarrow (iii)$. We have that $\ee[S^{\bm{Q}}_{i\leftarrow j}(u)]>0$ implies that  $\pp(S^{\bm{Q}}_{i\leftarrow j}(u)>0)$ is a positive probability, implying that a path must exist from component $j$ to $i$, which proves the result immediately.

\noindent
$(iii) \Rightarrow (i)$. Using the distributional equality of $S^{\bm{Q}}_{i\leftarrow j}(u)$, as given in (\ref{eq: distributional equality S_ij}), we have
\begin{align}
\label{eq: equality Expectation S_ij}
    \ee[S^{\bm{Q}}_{i\leftarrow j}(u)] = \bm{1}_{\{i=j\}} \pp(J_i>u) + \sum_{m=1}^d \ee[B_{mj}] \int_0^u g_{mj}(v)\ee[S^{\bm{Q}}_{i\leftarrow  m}(u-v)]\mathrm{d}v.
\end{align}
If $P_{i\leftarrow j} =1$, then the path must start with an edge $e_{kj}$, for some $k\in[d]$ and so $\ee[B_{kj}] > 0$ by definition.
If $k=i$, then $\ee[B_{ij}]>0$ and so this direct link implies $\ee[S^{\bm{Q}}_{i\leftarrow j}(u)]>0$.
If $k\neq i$, then we apply Eqn.\ (\ref{eq: equality Expectation S_ij}) to $\ee[S^{\bm{Q}}_{i\leftarrow k}(u-v)]$ on the RHS to obtain the next edge along the path.
Iterating this procedure for the non-zero terms on the RHS, we obtain a path from component $j$ to $i$, proving that $\ee[S^{\bm{Q}}_{i\leftarrow j}(u)]>0.$

\noindent
$(i) \Leftrightarrow (ii)$.
We can use the same argument, now for $\bm{\lambda}$.
We have that $\ee[S^{\bm{\lambda}}_{i\leftarrow j}(u)]$ satisfies
\begin{align}
    \ee[S^{\bm{\lambda}}_{i\leftarrow j}(u)] = \ee[B_{ij}] g_{ij}(u) + \sum_{m=1}^d \ee[B_{mj}] \int_0^u g_{mj}(v) \ee[S^{\bm{\lambda}}_{i\leftarrow  m}(u-v)]\mathrm{d}v,
\end{align}
due to the distributional equality for $S^{\bm{\lambda}}_{i\leftarrow j}(u)$ that is stated in (\ref{eq: distributional equality S_ij}).
A similar reasoning as above yields the result.
\end{proof}

\end{document}